\documentclass[a4paper, draft, 11pt]{article}
\usepackage[ngerman, english]{babel}
\usepackage[margin=2.5cm]{geometry}
\usepackage[T1]{fontenc}
\usepackage{anyfontsize}
\usepackage{physics,mathtools,bm}
\usepackage[overload]{textcase}                                       
\usepackage[shortlabels]{enumitem}                                    
\usepackage{etoolbox}                                                 
\usepackage{microtype}                                                
\allowdisplaybreaks[4]                                                
\usepackage{aligned-overset}                                          
\usepackage{stmaryrd}                                                 
\numberwithin{equation}{section}                                      
\usepackage{titlesec}
\usepackage{amsthm}
\usepackage{amsmath}
\usepackage[xcharter, varbb, osf]{newtx} 
\usepackage{thmtools}
\usepackage[hidelinks]{hyperref}
\usepackage{bookmark}

\makeatletter
\def\author#1{\gdef\@author{\MakeUppercase{#1}}}
\let\case@thanks\thanks
\def\thanks#1{\NoCaseChange{\case@thanks{#1}}}
\makeatother

\makeatletter
\renewcommand{\maketitle}{
  \begin{center}
    \vspace*{2em}
    {\bfseries\MakeUppercase \@title \par}
    \vspace{2em}
    {\@author \par}
    \vspace{0.5em}
    {\small \@date \par}
  \end{center}
  \@thanks
}
\makeatother

\titleformat{\section}
  {\scshape\center}
  {\thesection.}{0.5em}
  {}

\titleformat{\subsection}[runin]
  {\normalfont\normalsize}
  {\thesubsection}{0.5em}{\bfseries}
  [.]

\titlespacing*{\section}{0pt}{4.8ex plus 1.2ex minus 0.8ex}{1.8ex plus 0.4ex minus 0.4ex}
\titlespacing*{\title}{40pt}{8em}{1.8ex plus 0.4ex minus 0.4ex}

\makeatletter
\renewenvironment{abstract}{
    \small
    \quotation
    \noindent
    \textsc{\abstractname.}\ \ignorespaces
}{
  \par\vspace{2em}
}
\makeatother

\newtheoremstyle{theorem}{6pt}{6pt}{\itshape}{}{\bfseries}{.}{0.5em}{}
\newtheoremstyle{note}{6pt}{6pt}{\upshape}{}{\bfseries}{.}{0.5em}{}
\theoremstyle{theorem}
\newtheorem{proposition}{Proposition}[section]
\newtheorem{theorem}[proposition]{Theorem}
\newtheorem{lemma}[proposition]{Lemma}
\newtheorem{corollary}[proposition]{Corollary}

\theoremstyle{note}
\newtheorem{definition}[proposition]{Definition}

\newlist{myclaims}{enumerate}{3}
\setlist[myclaims,1]{label=\textbf{Claim \arabic*.}, align=left, leftmargin=2em, labelindent=0em, listparindent=0em, labelsep=1em, labelwidth=0em, ref= \arabic*}
\setlist[myclaims,2]{label=\textbf{Claim~\arabic{myclaimsi}.\arabic*.}, align=left, leftmargin=2em, labelindent=0em, listparindent=0em, labelsep=1em, labelwidth=0em, ref=\arabic{myclaimsi}.\arabic*}
\setlist[myclaims,3]{label=\textbf{Claim~\arabic{myclaimsi}.\arabic{myclaimsii}.\arabic*.}, align=left, leftmargin=2em, labelindent=0em, listparindent=0em, labelsep=1em, labelwidth=0em, ref=\arabic{myclaimsi}.\arabic{myclaimsii}.\arabic*}

\newlist{mycases}{enumerate}{3}
\setlist[mycases,1]{label=\textsc{Case \arabic*.}, align=left, leftmargin=2em, labelindent=0em, listparindent=0em, labelsep=1em, labelwidth=0em, ref=\arabic*}
\setlist[mycases,2]{label=\textsc{Case~\arabic{mycasesi}.\arabic*.}, align=left, leftmargin=2em, labelindent=0em, listparindent=0em, labelsep=1em, labelwidth=0em, ref=\arabic{mycasesi}.\arabic*}
\setlist[mycases,3]{label=\textsc{Case~\arabic{mycasesi}.\arabic{mycasesii}.\arabic*.}, align=left, leftmargin=2em, labelindent=0em, listparindent=0em, labelsep=1em, labelwidth=0em, ref=\arabic{mycasesi}.\arabic{mycasesii}.\arabic*}

\newlist{mylist}{enumerate}{3}
\setlist[mylist]{label*=\Roman*, align=left, leftmargin=2em, labelindent=0em, listparindent=0em, labelsep=1em, labelwidth=0em}

\begin{document}

\author{Anita Behme\thanks{TU Dresden,
  Institut f\"ur Mathematische Stochastik \& Center for Scalable Data Analytics and Artificial Intelligence (ScaDS.AI), 01062 Dresden, Germany, e-mail: anita.behme@tu-dresden.de}, Markus Riedle\thanks{King's College London, London Probability at King's, Strand, WC2R 2LS London, email: markus.riedle@kcl.ac.uk}, and Shend Thaqi\thanks{TU Dresden,
Institut f\"ur Mathematische Stochastik, 01062 Dresden, Germany, e-mail: shend.thaqi@tu-dresden.de}}

\title{Regularity of multiplicative processes on infinite-dimensional Lie groups}
\maketitle
\begin{abstract}
  This article studies regularity properties of multiplicative stochastic processes on infinite-dimensional Lie groups. We investigate conditions under which these processes admit c\`adl\`ag modifications and derive bounds on their local behavior. Our approach builds on the local equivalence of Banach-Lie groups and Banach spaces via the exponential and logarithm, allowing us to transfer analytic estimates and structural results.   
\end{abstract}



\section{Introduction}	
Infinite-dimensional analogues of Lie groups naturally emerge in many areas of mathematics and physics, including groups of invertible operators, gauge groups, loop groups, and diffeomorphism groups of manifolds. Such groups serve as fundamental symmetry structures in mathematical physics, particularly in gauge theory and quantum field theory~\cite{cheng2000gauge}, and in hydrodynamics, where they describe configuration spaces of fluid flows~\cite{abraham2008foundations}. More generally, they provide a natural framework for modelling dynamical systems with infinite-dimensional symmetries. A key step toward understanding such systems is to incorporate random perturbations, leading to the study of stochastic processes taking values in infinite-dimensional Lie groups.

In this article, we investigate the regularity properties of Banach-Lie group-valued stochastic processes with independent increments. Such processes arise, for example, as solutions to stochastic differential equations driven by vector space-valued noise~\cite{estrade1992exponentielle, hakim2006exponentielle}, or as natural models of random evolution in systems subject to symmetry constraints~\cite{skorokhod1982operator}. Our focus is on establishing conditions under which these processes admit regular modifications, that is, versions with c\`adl\`ag or continuous sample paths, and on deriving quantitative bounds for their expected displacement. 

While these questions are classical in the setting of normed vector spaces, the transition to Banach-Lie groups introduces two major challenges: the absence of a linear structure and the non-commutativity of the group operation. Consequently, fundamental notions such as “distance from the origin” or “increment size” require careful reinterpretation via invariant metrics and local coordinates. We address these difficulties by analyzing the local behaviour of multiplicative processes, exploiting the exponential and logarithm maps to transfer results from the linear to the non-linear setting.

The c\`adl\`ag regularity of stochastic processes on Lie groups has been well established in the finite-dimensional setting. Hunt’s seminal work on convolution semigroups on locally compact groups, and later developments by Heyer and Liao~\cite{hunt1956semigroups, heyer1977probability, liao2004levy}, show that Lévy processes on such groups are Hunt processes and therefore admit c\`adl\`ag modifications. In the broader context of manifold-valued semimartingales, the general theory of \'Emery~\cite{emery1989stochastic} ensures the existence of c\`adl\`ag sample paths for jump semimartingales in local charts, while Kunita’s stochastic flow theory~\cite{kunita1990stochastic} extends this to solutions of stochastic differential equations on smooth manifolds. For infinite-dimensional Lie groups, c\`adl\`ag regularity has mainly been studied through Lévy-driven stochastic flows on diffeomorphism groups, notably in the work of Fujiwara and Kunita~\cite{fujiwara1985stochastic} and subsequent extensions by Cruzeiro and Malliavin~\cite{cruzeiro1996renormalized}. However, these analyses rely on semimartingale or diffusion assumptions and on the geometric structure of diffeomorphism groups.

In contrast, Banach-Lie group-valued processes with independent increments have received far less attention. To the best of our knowledge, no general framework has been developed that establishes c\`adl\`ag regularity directly from the algebraic and analytic properties of Banach-Lie groups themselves. The present work addresses this gap by deriving regularity results for multiplicative processes using local exponential-logarithm representations, thereby extending the finite-dimensional L\'evy-type regularity theory to the Banach-Lie setting.

This perspective highlights the fact that, locally, Banach-Lie groups resemble Banach spaces, allowing a systematic extension of vector-space techniques to the realm of smooth manifolds and infinite-dimensional Lie groups.

Our main results include general regularity theorems for multiplicative processes on Banach-Lie groups. In Section~\ref{Preliminaries}, we state the definition of a Banach-Lie group, properties that are needed for our study, and we define multiplicative processes. In Section~\ref{Regularity}, the existence of a c\`adl\`ag modification is shown by studying the number of oscillations of a multiplicative process. We further prove in Section~\ref{Results for multiplicative processes with bounded jumps} that the moment generating function of \( d(0,x^{ s } _t) \) exists where \( d \) is a left- or right-invariant metric and \( (x^{ s } _t)_{ s \leq t }  \) is a multiplicative process with bounded jumps. We demonstrate in Section~\ref{Multiplicative processes on the Heisenberg group} how multiplicative processes can be constructed on infinite-dimensional nilpotent Lie algebras.

\section{Preliminaries}
\label{Preliminaries}

We study stochastic processes on infinite-dimensional Lie groups. Since we are interested in convergence, existence, and uniqueness, the Lie groups under consideration are assumed to be locally diffeomorphic to Banach spaces, i.e., Banach-Lie groups. Most of the Banach-Lie theoretic results we use are mentioned in~\cite{neeb2015towards} and the references therein. In finite dimensions, a Lie group is usually defined as a smooth manifold in a Euclidean space with group structure. This notion can be generalized by considering smooth manifolds in more abstract spaces, based on a compatible notion of differentiability~\cite{lang1999fundamentals}. 

In Banach spaces, the two most important notions are G\^{a}teaux differentiability and Fr\'echet differentiability, but only the latter characterizes the derivative as a bounded linear operator, which is crucial~\cite{zorn1945characterization}. Hence, whenever we refer to a \emph{differentiable} mapping, we mean a Fr\'echet differentiable mapping, and a \emph{smooth} mapping is understood as one that is arbitrarily often Fr\'echet differentiable. Based on this, we use Definition I.3.1 of a \emph{Banach manifold} given in~\cite{neeb2015towards}.
\begin{definition}
  We say \( \mathscr{G}  \) is a \emph{Banach-Lie group} if it is a Banach manifold equipped with a multiplication map~\( m : \mathscr{G}  \times \mathscr{G}  \to \mathscr{G}  , \, (g,h) \mapsto gh \) which is smooth, an \emph{identity (element)}~\( e \in \mathscr{G}  \) for which \( eg= ge = g \) for all~\( g \in \mathscr{G}  \), and a smooth inversion map~\( i: \mathscr{G} \to \mathscr{G} , \, g \mapsto g^{ -1 }   \), where \( g^{ -1 }  \) satisfies~\( g^{ -1 } g = g g^{ -1 } = e \). 
\end{definition}
A \emph{Banach-Lie algebra} is a Banach space~\( \mathfrak{g}  \) equipped with a \emph{Lie bracket}, i.e., a map~\( [ \,\cdot\, , \,\cdot\, ]: \mathfrak{g} \times \mathfrak{g} \to \mathfrak{g}   \) satisfying
\begin{enumerate}[1)]
  \item \( [ \,\cdot\, , \,\cdot\, ] \) is bilinear and continuous,
  \item for all~\( G,H \in \mathfrak{g}  \), we have \( [G,H] = -[H,G] \),
  \item for all~\( F,G,H \in \mathfrak{g} \), the \emph{Jacobi identity} holds, i.e.,
    \begin{equation*}
      [F,[G,H] ] + [G,[H,F]] + [H,[F,G] ] = 0.
    \end{equation*}
\end{enumerate}
In finite dimensions, there is a correspondence between Lie groups and Lie algebras (see~\cite{varadarajan2013lie}): For each finite-dimensional Lie group, its tangent space at the identity defines a Lie algebra, and conversely, every finite-dimensional Lie algebra arises as the tangent space at the identity of a Lie group. In infinite dimensions, the latter statement fails: There exist Banach-Lie algebras that are not associated with any Banach-Lie group~\cite{est1964non}. However, starting with a Banach-Lie group, one can always construct its associated Lie algebra. The following result collects this and other important properties of Banach-Lie groups, as presented in Chapter II of~\cite{neeb2015towards}.
\begin{proposition}
  Suppose \( \mathscr{G}  \) is a Banach-Lie group. Then the tangent space~\( \mathfrak{g} := T_e \mathscr{G}\) at the identity, is a Banach-Lie algebra. Furthermore, there exists a smooth map denoted~\( \operatorname{exp} : \mathfrak{g} \to \mathscr{G}  \) called \emph{exponential} such that \( \exp(0 ) = e \). Moreover, there are open sets~\( U_{ \mathfrak{g} }  \subseteq  \mathfrak{g}, \, U_{ \mathscr{G}} \subseteq \mathscr{G}   \) with \( 0 \in U_{ \mathfrak{g} } , \, e \in U_{ \mathscr{G}  }  \), such that \( \exp\vert_{U_{ \mathfrak{g} } } :\, U_{ \mathfrak{g} } \to U_{ \mathscr{G}  } \) is diffeomorphic.
\end{proposition}
We call the inverse map~\( \operatorname{log} : =(\exp\vert_{ U_{ \mathfrak{g} }  } )^{ -1 } : U_{ \mathscr{G}  } \to U_{ \mathfrak{g} }   \) the~\emph{logarithm}. The exponential, logarithm and multiplication are smooth, hence the map
\begin{align}
  \{ (X,Y): X,Y \in U_{ \mathfrak{g} } , \, \exp (X)\exp (Y) \in U_{ \mathscr{G}  } \} &\to \mathfrak{g}, \nonumber\\
  (X,Y) &\mapsto \operatorname{log} ( \exp (X) \exp (Y) ) \label{BCH function}
\end{align}
is also smooth. Hence, the domain in~(\ref{BCH function}) is open as an intersection of open sets and nonempty as \( ( 0,0) \) is contained. Therefore, 
\begin{equation*}
  \lim_{ \delta \to 0 }  \sup_{ \norm{ X } , \norm{ Y } < \delta  } \norm{ \operatorname{log} ( \exp ( X) \exp (Y) ) } = 0.
\end{equation*} 
We can fix \( \rho' > 0 \) such that the open ball centered at~\( 0 \in \mathfrak{g} \) with radius~\( \rho' \), 
\[ B_{ \rho' } := \{ G \in \mathfrak{g} : \norm{ G } < \rho' \},\]
is contained in the open set~\( U_{ \mathfrak{g} } \). We can also fix \( \rho'' \in (0, \rho') \) such that the domain in~(\ref{BCH function}) contains \( B_{ \rho'' } \times B_{ \rho'' } \). In fact, for every~\( \delta_1 < \rho' \), there exists \( \delta_2 < \rho' \) such that \( B_{ \delta_1 } \times B_{ \delta_2 } \) is contained in the domain of (\ref{BCH function}). These properties and the parameters~\( \rho', \rho'' \) will be used extensively throughout this work.

\subsection{Stochastic processes on Lie groups}
\label{Stochastic processes on Lie groups}
From now on, we assume that \( \mathscr{G}  \) is a separable Banach-Lie group with associated Lie algebra~\( \mathfrak{ g }  \). Then \( \mathfrak{g} \) is obviously separable as well. Therefore, we can introduce the concept of measurability and random variables with values in~\( \mathscr{G}  \). 

Throughout this work, \( (\Omega,\mathscr{F},\mathbb{P}) \) denotes a complete probability space. A map~\( x:\Omega\to\mathscr{G} \) is called a \emph{random element} in~\( \mathscr{G} \) if \( x^{-1}(A)\in\mathscr{F} \) for every open set~\( A\subseteq\mathscr{G} \) (equivalently, \(x\) is measurable with respect to the Borel-\(\sigma\)-algebra on~\(\mathscr{G}\)). Since multiplication and inversion are continuous (hence Borel-measurable) and \( \mathscr{G} \) is separable, products and inverses of random elements are again random elements. In what follows, \( ( \mathscr{ F } ^{ s } _t)_{ 0 \leq s \leq t < \infty  }  \) shall be an \emph{additive filtration}, i.e., a family of~\( \sigma\)-algebras~\( \mathscr{F}^s_t\subseteq\mathscr{F} \) such that for all~\( s\leq t\leq u\leq v \) the following hold:
\begin{enumerate}[1)]
  \item \( \mathscr{F}^s_t \) and \( \mathscr{F}^u_v \) are independent,
  \item \( \mathscr{F}^t_u \subseteq \mathscr{F}^s_v \).
\end{enumerate}
We further assume that the \emph{usual hypotheses} hold; that is, for fixed~\( s \geq 0 \), the filtration~\( (\mathscr{F}^{s}_{t})_{t \geq s} \) is right-continuous in~\( t \), and for all~\( t \geq s \), the~\( \sigma \)-algebra~\( \mathscr{F}^{s}_{t} \) is complete. One may start with a family~\( (\mathscr{F}^{s}_{t})_{s \leq t} \) that satisfies neither of the usual hypotheses. This is not a problem, since the augmentation~\( (\overline{\mathscr{F}}^{s}_{t})_{t \geq s} \) for fixed~\( s \) remains additive (see Chapter~1.1 of~\cite{ito2013stochastic}).

We define a \emph{one-parameter stochastic process} in~\( \mathscr{G} \) as a family~\( (x_{t})_{0 \leq t < \infty} \) of random elements. A \emph{two-parameter stochastic process} is a family of the form~\( (x^{s}_{t})_{0 \leq s \leq t < \infty} \) of random elements. We will usually write \( (x^{s}_{t})_{s \leq t} \) and \( (\mathscr{F}^{s}_{t})_{s \leq t} \) instead, since the index set is always understood to be 
\[\{(s,t) \in [0,\infty)^{2} :\, s \leq t\}. \]
A process~\( (x^{s}_{t})_{s \leq t} \) is called \emph{adapted to~\( (\mathscr{F}^{s}_{t})_{s \leq t} \)} if \( x^{s}_{t} \) is measurable with respect to~\( \mathscr{F}^{s}_{t} \) for all~\( s \leq t \). We say \( x \) is \emph{stochastically continuous} in \( (s,t) \) if
  \begin{equation*} 
    \lim_{ h,k \to 0} \mathbb{ P } {\left( (x^{ s }_t)^{ -1 } x^{ s+h } _{ t+k } \in U  \right)} = 1 \qquad \text{for all open sets~\( U \subseteq \mathscr{G} \) containing~\( e \)}. 
  \end{equation*}
We do not restrict the sign of the null sequences~\( h,k \), but assume throughout that~\( s+h, t+k \geq 0 \). A process~\( x \) is called \emph{stochastically continuous} if it is stochastically continuous in all pairs~\( s \leq  t \). Taking the topology on \( \mathscr{ G }  \) into account, this is equivalent to \( x^{ s+h } _{ t+k } \to x^{ s } _t \) in probability.
\begin{definition}
  \label{Multiplicative process}
  A two-parameter stochastic process~\( x= (x^{ s } _t) _{ s \leq t }  \) is called a \emph{multiplicative process} with respect to~\( (\mathscr{ F } ^{ s } _t)_{ s \leq t }  \) if
  \begin{enumerate}[1)]
    \item the process~\( x \) is adapted to \( (\mathscr{ F } ^{ s } _t)_{ s \leq t }  \),
    \item the process~\( x \) is \emph{multiplicative}, i.e., \( x^{ s } _s = e \) and \( x^{ s } _t x^{ t }_u = x^{ s } _u  \) almost surely for all~\( s \leq t \leq u \), \label{multiplicativity order}
    \item the process~\( x \) is stochastically continuous.
  \end{enumerate}
\end{definition}
One may reverse the order of multiplication in~\ref{multiplicativity order} of Definition~\ref{Multiplicative process}, so that \( x^{t}_{u} x^{s}_{t} = x^{s}_{u} \), and adjust our results accordingly. In that case, \( x^{-1} = \big( (x^{s}_{t})^{-1} \big)_{s \leq t} \) is multiplicative from left to right. 

In practice, most multiplicative processes appear as time-ordered exponentials, i.e., stochastic exponentials with respect to It\^o, Stratonovich, and Marcus differentials, of additive processes~\cite{skorokhod1982operator}. Conventionally, a process on Lie groups is called a \emph{semimartingale} if it admits a \emph{stochastic logarithm} in the associated Lie algebra. It is known that if a process \( x \) on a finite-dimensional Lie group is both, multiplicative and a semimartingale, then it is the stochastic exponential of an additive process~\cite{estrade1992exponentielle, hakim2006exponentielle}. However, our work so far still contributes in important ways. First, we are only aware of such results for the finite-dimensional case. Second, even if finite-dimensional Lie groups are considered, it is not yet clear which multiplicative processes are semimartingales. Therefore, given a multiplicative process, it is not possible to consider the stochastic logarithm, which would be an additive process that lives in a finite-dimensional Hilbert space and hence would be a process with very convenient properties~\cite{gikhman2004theory, ito2013stochastic}. To the best of our knowledge, it has not even been established that L\'evy processes on arbitrary finite-dimensional Lie groups are semimartingales. This is mentioned in~\cite{kovalchuk1993semimartingale}.

In the one-dimensional case, given an additive process~\( (X_t)_t \), a multiplicative process can be constructed by setting~\( x_t = \exp ( X_t) \). In higher-dimensional matrix spaces and other Lie algebras, this construction does not work, except if \( X_s \) and \( X_t \) commute almost surely. Otherwise, an error arises between~\( \exp ( X_t ) \) and \( \exp(X_t - X_s)\exp (X_s)  \). In fact, if \( U,V \in \mathfrak{g}\) are sufficiently small in norm, then the Baker-Campbell-Hausdorff formula~\cite{dynkin1947calculation, dynkin1950normed} holds, i.e.,
\begin{equation*}
    \exp ( U) \exp ( V) = \exp (U + V + \tfrac{ 1 }{ 2 } [U,V] + \tfrac{ 1 }{ 12 } {\left( [U, [U,V]]  - [V , [U,V]] \right)} + \dots ) ,
\end{equation*}
where the right-hand side contains a series of nested commutators. For nilpotent Lie algebras, the series is eventually cut off. In particular, the commutator of the Lie algebra associated with the Heisenberg group satisfies \( [U,[U,V]] = [V,[U,V]] =0 \). Therefore, the Baker-Campbell-Hausdorff formula simplifies to 
\begin{equation*}
    \exp ( U) \exp ( V) = \exp (U + V + \tfrac{ 1 }{ 2 } [U,V]).
\end{equation*}
This fact is used to show that on the Lie algebra of the Heisenberg group, we have
\begin{align*}
    \lim_{ \max \Delta_{ n } \to 0 } \prod_{ n } \exp ( \Delta _n X) &= \lim_{ \max \Delta _n \to 0 } \exp ( \sum_{ n } \Delta_{ n } X + \tfrac{ 1 }{ 2 } \sum_{ n<m }  [\Delta_{ n } X,\Delta_{ m } X]) \\
                                                                     &= \exp ( \int_{ 0 < t_1 \leq  t } \dd X_{ t_1 }  + \tfrac{ 1 }{ 2 } \int_{ 0<t_1<t_2 \leq t } [ \dd X_{ t_1 } , \dd X_{ t_2 } ] ).
\end{align*}
As a result, we can explicitly compute a multiplicative process on the Heisenberg group using an additive process on the Lie algebra. An example of such a construction is given in Section \ref{Multiplicative processes on the Heisenberg group}.

In~\cite{behme2012multivariate}, it is shown that multiplicative processes can be used to explicitly express solutions of generalized Ornstein-Uhlenbeck SDEs in higher dimensions. Furthermore, an equivalence criterion is stated, characterizing whether the resulting Ornstein-Uhlenbeck process has a stationary distribution based on properties of the multiplicative process.

Before we continue with properties of multiplicative processes, we introduce a local notion of a norm on~\( \mathscr{G}  \). For brevity, for open sets~\( U \subseteq \mathscr{G} \) containing \( e \), we set
\begin{equation}
  U^{ 0 } = \{e\}, \quad U^{ n } = \{ g_1 \cdots g_n :\, g_1, \dots, g_n \in U \}. \label{Cover of steps}
\end{equation}
The next lemma records some elementary properties of this construction.
\begin{lemma}
  \label{Local metric of Lie groups}
  Suppose \( \delta < \rho ' \) and set~\( U := \operatorname{exp}{\left( B_{ \delta  }  \right)}  \). Then, for each~\( j,k \in \mathbb{ N }  \) with~\( j \leq k \), we have:
  \begin{enumerate}[1)]
    \item if \( g \in U ^{ j }  \), then \( g^{ -1 } \in U ^{ j }  \),
    \item if \( g \in U^{ j } , \, h \in U ^{ k }  \), then \(  gh, hg \in U^{ j+k }  \),
  \end{enumerate}
\end{lemma}
Lemma~\ref{Local metric of Lie groups} is crucial to show convergence results on Lie groups. On normed vector spaces, the properties listed in the lemma above simplify to properties of the norm. If \( \mathscr{G}   \) were commutative, then its connected component containing~\( e \) could be identified with its Lie algebra and the exponential would act as the identity. In that case, \( g \in U_{ \delta  } ^{ j } \equiv B^{ j } _{ \delta  } = B_{ j \delta  }  \), i.e., \( \norm{ g } \leq j \delta  \). 

On non-commutative groups, these inequalities only hold for open sets of the form~\( U_{ \delta  } = \operatorname{exp} ( B _{ \delta  } ) \) for small~\( \delta  \). This is not a problem as the topology on~\( \mathscr{G}  \) is induced by the system 
\begin{equation}
  \tau = \{ g \exp ( B_{ \delta  }) :\, g \in \mathscr{G}, \delta \in (0, \rho ' )\}, \label{Symmetric system}
\end{equation}
where we define sets of the form~\( gU := \{ gh : h \in U \} \). Therefore, it is sufficient to show convergence with open sets of the form~\( U_{ \delta  }  \) as for any open set~\( U \) with~\( e \in U\), there is \( \delta > 0 \) such that \( U_{ \delta  } \subseteq U \). 

For a process~\( x= (x^{ s } _t)_{ s \leq t }  \) satisfying the first two conditions in Definition~\ref{Multiplicative process}, we have
\begin{equation*}
  (x^{ s } _t)  ^{ -1 } x^{ s+h } _{ t+k } = (x^{ s+h } _{ t})^{ -1 } (x^{ s } _{ s+h } )^{ -1 } x^{ s+h } _t x^{ t } _{ t+k } .
\end{equation*}
Hence, using Lemma~\ref{Local metric of Lie groups}, if \( x^{ t } _{ t+k } \to e \) and \( x^{ s }_{ s+h } \to e  \) in probability, then also \( (x^{ s } _{ s+h } )^{ -1 } \to e \), and the process would be stochastically continuous. We find that a process that already satisfies the first two conditions of a multiplicative process is stochastically continuous if and only if
\begin{equation*}
  \lim_{ h \to 0} \mathbb{ P } ( (x_{ t } )^{ -1 } x _{ t+h } \in U ) = 1 \qquad \text{for all~\( t> 0 \) and \( U \) open with~\( e \in U \)} .
\end{equation*}
Again, we assume \( t+h \geq 0 \). Not only can we find conditions that are easy to check to obtain stochastic continuity. But stochastic continuity also implies a locally uniform version of stochastic continuity, which we call \emph{uniform stochastic continuity}. The proof uses a classical compactness argument.
\begin{lemma}
  \label{Uniform stochastic continuity}
  Let \( x = (x_t)_t \) be a stochastically continuous process and \( U \) an open set containing~\( e \). Then, for all~\( T , \alpha > 0 \), there is \( H > 0 \) such that \( \sup_{ \substack{ \abs{ h } < H \\ t < T }  } \mathbb{ P } {\left( (x_t)^{ -1 } x_{ t+h } \notin U \right)}  \leq  \alpha \).
\end{lemma}

\section{Regularity}\label{Regularity} 
\subsection{A modification with c\`adl\`ag paths}

Skorokhod~\cite{skorokhod1982operator} referred to processes similarly defined as those in Definition~\ref{Multiplicative process} as \emph{stochastic semigroups}. We refer to these as multiplicative processes to emphasize their formal resemblance to additive processes~\cite{gikhman2004theory, ito2013stochastic}. The notion of additive processes generalizes to separable Banach spaces rather intuitively. On Banach-Lie groups, the binary operation is no longer commutative and the existence of a global norm is not guaranteed. Our first endeavor is to show that multiplicative processes have a modification with c\`adl\`ag paths. Continuity is a local property. Since multiplication in a Banach-Lie group locally resembles addition, and the group is equipped with a local metric, we can manage without global commutativity or a globally defined norm. 

We first state some properties for products~\( x_1 \cdots x_n \) of independent random elements in~\( \mathscr{G}  \). These results are multiplicative versions of the results shown in~\cite{gikhman2004theory} and~\cite{ito2013stochastic}. 
\begin{lemma}
  \label{Maximum oscillation bound}
     For \( n \in \mathbb{ N }  \), let \( x_1, \dots, x_n \) be independent random elements in~\( \mathscr{G}  \). Define \( x^{ j } _k = x_{ j+1 } \cdots x_k \) for~\( j \leq k \leq n\). Suppose for all~\( j \leq n \) and some~\( \delta \in (0, \rho ') , \alpha > 0 \), we have \( \mathbb{ P } (x^{ 0 } _j \notin U_{ \delta  } ) \leq \alpha   \). Then
  \begin{equation*}
     ( 1- \alpha ) \mathbb{ P } (\exists_{ j \leq n }:\,  x^{ j } _n  \notin U^{ 2 } _{ \delta  }) \leq \mathbb{ P } ( x^{ 0 } _n \notin U_{ \delta  } ) .
  \end{equation*}
\end{lemma}
\begin{proof}
  We first show
  \begin{equation}
    \bigcup_{ j \leq n } \bigcap_{ k > j } \{ x^{ 0 } _j \in U_{ \delta  } ,  x^{ j } _n \notin U^{ 2 } _{ \delta  } , x^{ k } _n \in U^2_{ \delta  }  \} \subseteq \{ x^{ 0 } _n \notin  U_{ \delta  } \}  . \label{Maximum oscillation bound-1}
  \end{equation}
  If \( x^{ 0 } _j \in U_{ \delta  } \) and \( x^{ j } _n \notin U^{ 2 } _{ \delta  }  \), then \( x^{ 0 } _n \notin U_{ \delta  }  \) by Lemma~\ref{Local metric of Lie groups}. Therefore,
  \begin{equation*}
    \{ x^{ 0 } _j \in U_{ \delta  } , x^{ j } _{ n } \notin U^{ 2 } _{ \delta  }  \}  \subseteq  \{ x^{ 0 } _n \notin U_{ \delta  }  \},
  \end{equation*}
  This shows (\ref{Maximum oscillation bound-1}). Notice that the left-hand side is a union of disjoint sets. It is easy to see that
  \begin{equation}
    \{ \exists_{ j \leq n } : x^{ j } _n \notin U^{ 2 } _{ \delta  }  \} = \bigcup_{ j \leq n } \bigcap_{ k >  j } \{ x^{ j } _n \notin U^{ 2 } _{ \delta  } , x^{ k } _n \in U_{ \delta  } ^{ 2 }  \} . \label{Maximum oscillation bound-2}
    \end{equation}
  Now we are in a position to prove the lemma. We have
  \begin{align*}
    \mathbb{ P } (x^{ 0 } _n \notin U_{ \delta  } ) \overset{\text{(i)}}&{\geq} \mathbb{ P } {\biggl( \bigcup_{ j \leq n } \bigcap_{ k >  j }  \{ x^{ 0 } _j \in U_{ \delta  } , x^{ j } _{ n } \notin U^{ 2 } _{ \delta  } , x^{ k } _{ n } \in U^{ 2 } _{ \delta  }  \} \biggr)}           \\
    \overset{\text{(ii)}}&{=}  \sum_{ j \leq n } \mathbb{ P } {\biggl( \bigcap_{ k >  j }  \{ x^{ 0 } _j \in U_{ \delta  } , x^{ j } _{ n } \notin U^{ 2 } _{ \delta  } , x^{ k } _{ n } \in U^{ 2 } _{ \delta  } \} \biggr) }          \\
    \overset{\text{(iii)}}&{=} \sum_{ j \leq n } \mathbb{ P } {\left( \{ x^{ 0 } _j \in U_{ \delta  } \} \right)} \cdot \mathbb{ P } {\biggl( \bigcap_{ k >  j }  \{ x^{ j } _{ n } \notin U^{ 2 } _{ \delta  } , x^{ k } _{ n } \in U^{ 2 } _{ \delta  } \} \biggr)}          \\
    \overset{\text{(iv)}}&{\geq} (1 - \alpha ) \sum_{ j \leq n } \mathbb{ P } {\biggl( \bigcap_{ k > j } \{ x^{ j } _n \notin U^{ 2 } _{ \delta  } , x^{ k } _{ n } \in U^{ 2 } _{ \delta  } \} \biggr)} \\
    \overset{\text{(v)}}&{=} (1- \alpha ) \mathbb{ P } {\left( \exists_{ j \leq n } :\,  x^{ j } _n \notin U^2_{ \delta  }  \right)}.  
  \end{align*}
  Here, (i) holds by (\ref{Maximum oscillation bound-1}), (ii) as (\ref{Maximum oscillation bound-1}) is a disjoint union of sets. In (iii), we use the independence of the random elements. (iv) is valid by the assumption in the assertion. (v) holds due to (\ref{Maximum oscillation bound-2}) and the fact that the union is over disjoint sets.
\end{proof}
\begin{lemma} 
  \label{Largest step bound}
  For \( n \in \mathbb{ N }  \), let \( x_1, \dots, x_n \) be independent random elements in~\( \mathscr{G}  \). Define \( x^{ j } _{ k } = x_{ j+1 } \cdots x_k \) for~\( j \leq k \leq n \). Then, for all~\( \delta \in ( 0 , \rho ')\), we have
  \begin{equation*}
    \mathbb{ P } (\exists_{ j<k \leq n } :\,  x^{ j } _k \notin U^2_{ \delta  } ) \leq \mathbb{ P } ( \exists_{ j \leq n } :\, x^{ j } _n \notin U_{ \delta  } ) .
  \end{equation*}
\end{lemma}
\begin{proof}
  By Lemma~\ref{Local metric of Lie groups} 2), the right-hand event is larger than the left-hand one.
\end{proof}
We can define a notion of oscillations for maps on a Lie group similar to real-valued functions although we do not have an explicit norm to work with. We only consider oscillations that are "small enough". To be precise, let \( \delta > 0 \) such that \( U_{ \delta  } = \operatorname{exp} ( B _{ \delta  } ) \subseteq \operatorname{dom} (\operatorname{log})   \). For finite~\( \mathbb{ T } = \{ t_0 < \dots < t_n \}  \), we define the \emph{number of oscillations} larger than~\( \delta  \) of~\( x \) on~\( \mathbb{ T }  \) as the random variable
\begin{equation}
  \bm{ O } _x^{ \delta  } ( \mathbb{ T } ) = \max \{ m \in \mathbb{ N } :\, \text{there are \( \tau _0 < \dots < \tau_{ m } \in \mathbb{ T }  \) such that \( x^{ \tau_{ j-1 }  }_{ \tau_{ j }  } \notin U_{ \delta  } \}  \)} . \label{Oscillation counter 1}
\end{equation}
For countably infinite~\( \mathbb{ T }  \), we define
\begin{equation}
  \bm{ O } _x^{ \delta  } ( \mathbb{ T } ) = \sup \{ \bm{ O } _x^{ \delta  } ( \mathbb{ T } ' ) :\, \mathbb{ T } ' \subseteq \mathbb{ T } , \, \text{\( \mathbb{ T } ' \) is finite} \}. \label{Oscillation counter 2}
\end{equation}
The following corollary states straightforward properties of this quantity:
\begin{corollary}
  \label{Properties of oscillation counter}
  For \(\mathbb{ S } , \mathbb{ T } \subseteq [0, \infty ) \) countable and \( \delta \in ( 0 , \rho ') \), the following assertions hold:
  \begin{enumerate}[1)]
    \item If \( \mathbb{ S }  \subseteq  \mathbb{ T }  \), then \( \bm{ O } _x^{ \delta  } ( \mathbb{ S } ) \leq \bm{ O } _x ^{ \delta  } ( \mathbb{ T } ) \). \label{Properties of oscillation counter 1)}
    \item If \( (t_m)_m \) is an exhaustive sequence in~\( \mathbb{ T }  \), then \( \lim_{ m \to \infty  } \bm{ O } _x^{ \delta  } ( \{ t_j \} _{ j \leq m } ) = \bm{ O } _x ^{ \delta  } ( \mathbb{ T } )   \).
    \item If for all~\( s \in \mathbb{ S }  \) and all~\( t \in \mathbb{ T }  \) it holds that \( s \leq t \), then \( \bm{ O } _x ^{ \delta  } ( \mathbb{ S } \cup \mathbb{ T } ) \leq \bm{ O } _{ x } ^{ \delta  } ( \mathbb{ S } ) + \bm{ O } _x ^{ \delta  } ( \mathbb{ T } ) + 1    .\) \label{Properties of oscillation counter 3)}
  \end{enumerate}
\end{corollary}
In order to show that there exists a c\`adl\`ag modification, we have to prove that the random variable~\( \bm{ O } _x^{ \delta  } ( \mathbb{ Q } \cap [0, T ])\) is finite almost surely for all sufficiently small~\( \delta >0 \) and all~\( T > 0 \). One way to guarantee a random variable is finite almost surely is to show that its expectation is finite.
\begin{lemma}
  \label{Expectation bound for oscillations}
   For~\( n \in \mathbb{ N }  \), let \( x_1, \dots, x_n \) be independent random elements in~\( \mathscr{G}  \). Define \( x^{ j } _{ k } = x_{ j+1 } \cdots x_k \) for~\( j \leq k \leq n \). Suppose that \( \mathbb{ P } ( \exists_{ j \leq k \leq n}:\, x^{ j } _k \notin U_{ \delta  } ) \leq \alpha < 1 \) for some~\( \delta \in  (0, \rho ') \). Then
  \begin{equation*}
    \mathbb{ E } {\left( \bm{ O } _x^{ \delta   }  ( \{ 1, \dots, n \} )\right)} \leq \frac{ \alpha  }{ 1- \alpha  } .
  \end{equation*}
\end{lemma}
\begin{proof}
  Denote \( \mathbb{ T } : = \{1,\dots,n\} \) and \( \mathbb{ T } ^{ j } _k = \{ j, \dots, k\} \). It is evident that the following inclusion of events holds for~\( m \in \mathbb{ N }  \):
  \begin{align}
      &{\left\{ \bm{ O } _x ^{ \delta  } ( \mathbb{ T } ) \geq m+1 \right\}} \nonumber\\
      & \qquad \subseteq  {\left\{  \exists _{ j < n } : \bm{ O } _x^{ \delta  } ( \mathbb{ T } ^{ 1 } _j ) \geq m, \,  \bm{ O } _x^{ \delta  } (  \mathbb{ T } ^{ j } _n ) = 1, \forall _{ k > j }: \bm{ O } _x^{ \delta  }  ( \mathbb{ T } ^k_n ) = 0 \right\}} \label{Expectation bound for oscillations (1)}.
  \end{align}
  It is further evident that the following equality of the events holds as well:
  \begin{equation}
    \{ \bm{ O }_x ^{ \delta  } ( \mathbb{ T } ) \geq 1 \} = \bigcup _{ j < n } {\left\{ \bm{ O } _x ^{ \delta  } ( \mathbb{ T } ^{ j } _n)  = 1, \forall _{ k > j } : \bm{ O } _x^{ \delta  } (\mathbb{ T } ^{ k } _n )  = 0 \right\}}  \label{Expectation bound for oscillations (3)}.
  \end{equation}
  We find therefore
  \begin{align}
      & \mathbb{ P } {\left( \bm{ O } _x ^{  \delta  } ( \mathbb{ T } ) \geq m+ 1\right)} \nonumber\\
      & \quad \overset{ \mathclap{(\ref{Expectation bound for oscillations (1)})} }{ \leq } \, \, \mathbb{ P } {\left(  \exists _{ j < n } : \bm{ O } _x^{ \delta  } ( \mathbb{ T } ^{ 1 } _j ) \geq m, \, \bm{ O } _x^{ \delta  } ( \mathbb{ T } ^{ j } _n ) = 1, \forall _{ k > j }: \bm{ O } _x ^{ \delta  } ( \mathbb{ T } ^{ k } _n ) = 0 \right)} \nonumber\\
      & \quad \leq \sum_{ j < n }  \mathbb{ P } {\left(  \bm{ O } _x^{ \delta  } ( \mathbb{ T } ^{ 1 } _j ) \geq m \right)} \cdot \mathbb{ P } {\left( \bm{ O } _x^{ \delta  } ( \mathbb{ T } ^{ j } _n ) = 1, \forall _{ k > j }: \bm{ O } _x ^{ \delta  } ( \mathbb{ T } ^{ k } _n ) = 0 \right)} \nonumber\\
      & \quad \overset{\mathclap{\text{(i)}}}{ \leq } \,  \mathbb{ P } {\left(  \bm{ O } _x^{ \delta  } ( \mathbb{ T }  ) \geq m \right)} \cdot \sum_{ j < n }   \mathbb{ P } {\left( \bm{ O } _x^{ \delta  } ( \mathbb{ T } ^{ j } _n ) = 1, \forall _{ k > j }: \bm{ O } _x ^{ \delta  } ( \mathbb{ T } ^{ k } _n ) = 0 \right)} \nonumber\\
      & \quad  \overset{ \mathclap{(\ref{Expectation bound for oscillations (3)})} }{ = }\, \,  \mathbb{ P } {\left(  \bm{ O } _x^{ \delta  } ( \mathbb{ T }  ) \geq m \right)} \cdot  \mathbb{ P } {\left(  \bm{ O }_x ^{ \delta  } ( \mathbb{ T } ) \geq 1 \right)},\label{Expectation bound for oscillations (4)}
  \end{align}
  where (i) holds as \( \bm{ O } _x ^{ \delta  } ( \mathbb{ T } ^{ 1 } _j ) \leq \bm{ O } _x ^{ \delta  } ( \mathbb{ T }  ) \), and noting that the right-hand side on (\ref{Expectation bound for oscillations (3)}) is a union of disjoint sets. 

  By recursion of (\ref{Expectation bound for oscillations (4)}), we find \( \mathbb{ P } {\left( \bm{ O } _x ^{  \delta  } ( \mathbb{ T } ) \geq m\right)} \leq \mathbb{ P } {\left( \bm{ O } _x ^{  \delta  } ( \mathbb{ T } ) \geq  1\right)}^{ m } \).
  Since \( \{\exists_{ j \leq k \leq n}:\, x^{ j } _k \notin U_{ \delta  } \}  = \{ \bm{ O } _{ x } ^{ \delta  } ( \mathbb{ T } ) \geq 1 \} \), we obtain using the assumption in the lemma,
  \begin{equation*}
    \mathbb{ E } (\bm{ O } _x^{ \delta   }  (\mathbb{ T } )) = \sum_{ m=1 } ^{ \infty  } \mathbb{ P } ( \bm{ O } _x^{ \delta   }  ( \mathbb{ T }) \geq m) \leq  \sum_{ m=1 } ^{ \infty  }  \alpha ^{ m } = \frac{ \alpha  }{ 1- \alpha  } . 
   \end{equation*}
\end{proof}
Our strategy to find a modification of~\( x \) with c\`adl\`ag paths is to show that the paths~\( t \mapsto x _t \) are \emph{regulated}, i.e., they have left and right limits almost surely. The following lemma gives an equivalence criterion for regulated functions in terms of oscillations. It is an immediate consequence of the fact that \( \{g U_{ \delta  } \}_{ g \in \mathscr{G} , \delta >0 }   \) induces the topology on~\( \mathscr{G}  \).
\begin{lemma}
  \label{Characterization of continuous functions}
  Let \( \mathbb{ T } \subseteq [0, \infty ) \) be countable. Suppose \( f: \mathbb{ T }  \to \mathscr{G}  \) is a map such that for each~\( \delta \in ( 0, \rho ') \) and~\( T >0 \), it holds \( \bm{ O } _f^{ \delta  } ( \mathbb{ T } \cap [0,T] ) < \infty \). Then \( f \) is regulated.
\end{lemma}
Notice that for any open neighborhood~\( U \subseteq \mathscr{G}  \) containing the identity, there is \( \delta > 0 \) sufficiently small such that \( U_{ \delta  } = \operatorname{exp} (B _{ \delta  })  \) is contained in~\( U \). To make use of Lemma~\ref{Maximum oscillation bound} and Lemma~\ref{Largest step bound}, we would like to ensure \( U_{ \delta  } ^{ 2 }  \subseteq U \) as well for small~\( \delta  \). But this already follows from the smoothness of the function in (\ref{BCH function}), which implies \( \lim_{ \delta \to 0 }  \sup_{ \norm{ X } , \norm{ Y } < \delta  } \norm{ \operatorname{log} ( \exp ( X) \exp (Y) ) } = 0  \), i.e., for every~\( \epsilon  \) there is \( \delta > 0 \) such that \( \sup_{ \norm{ X } , \norm{ Y } < \delta  } \norm{ \operatorname{log} ( \exp ( X) \exp (Y) ) } < \epsilon   \). Choose \( \epsilon  \) such that \( U_{ \epsilon  } \subseteq U \) and \( \delta  \) appropriately. Then we obtain
\begin{equation}
  U_{ \delta  }^2 \subseteq U_{ \epsilon  } \subseteq U \label{BCH formula I}.
\end{equation}
In consequence, a process~\( x = (x_t)_t \) is stochastically continuous if and only if
\begin{equation}
  \lim _{ h \to  0 } \mathbb{ P } {\left( (x_t)^{ -1 } x _{ t+h } \in U_{ \delta  } ^{ 2 }  \right)} = 1 \quad \text{for each \( \delta \in (0, \rho '' ) \)} . \label{Stochastic continuity 2}
\end{equation}
We can now show the existence of a c\`adl\`ag modification.
\begin{theorem}
  \label{Regularization of multiplicative processes}
  Every~\( \mathscr{ G }  \)-valued multiplicative process~\( x = (x^{ s } _t)_{ s,t }  \) with respect to~\( (\mathscr{ F } ^{ s } _{ t } )_{ s,t }  \) has a modification~\( \overline{ x } = ( \overline{ x } ^{ s } _t ) _{ s \leq t }  \) of~\( x \) such that
    \begin{enumerate}[1)]
      \item the process~\( \overline{ x }  \) is c\`adl\`ag in both arguments (simultaneously),
      \item the process~\( \overline{ x }  \) is a multiplicative process with respect to~\( ( \mathscr{ F } ^{ s } _t ) _{ s \leq t }  \),
      \item for all~\( \omega \in \Omega  \) and~\( s \leq t \leq u \), it holds \( \overline{ x } ^{ s }_s( \omega ) = e   \) and \( \overline{ x } ^{ s } _u ( \omega ) = \overline{ x } ^{ s }_t ( \omega )  \overline{ x } ^{ t } _u ( \omega )   \).
    \end{enumerate}
\end{theorem}
\begin{proof}
  We show all claims for~\( 0 \leq s \leq t \leq u \leq T \), where \( T \) is fixed. Let \( \mathbb{ T } \subseteq [0,T] \) be countable and dense.
  \begin{myclaims}
    \item\label{Regularization of multiplicative processes I} For all~\( \alpha > 0 \) and~\( \delta \in ( 0 , \rho ') \), there is \( H > 0 \) such that for all~\( t_0 \leq  \dots \leq  t_n  \in \mathbb{ T } \), where \( \abs{ t_n - t_0 } \leq H \), 
      \begin{equation*}
        \mathbb{ P } {\left( \exists _{ j \leq k \leq n } : x^{ t_j } _{ t_k } \notin U_{ \delta  }  \right)} \leq \alpha .
      \end{equation*} 
    \item [\textit{Proof.}] By uniform stochastic continuity (Lemma~\ref{Uniform stochastic continuity}), for all~\( \beta > 0 \) and~\( \epsilon \in (0 , \rho')  \), there is \( H> 0 \) such that 
      \begin{equation*}
        \sup_{ \substack{ \abs{ h } < H \\ t+H \leq  T }  } \mathbb{ P } {\left( x^{ t } _{ t+h } \notin U_{ \epsilon  }  \right)} \leq \beta .
      \end{equation*}
      Hence, for all~\( t_0 \leq \dots \leq t_n \in \mathbb{ T }  \) where \( \abs{ t_n - t_0 } < H \), we have \( \mathbb{ P } ( x ^{ t_0 } _{ t_n } \notin U_{ \epsilon  } ) \leq \beta  \). Thus, Lemma~\ref{Maximum oscillation bound} applies, and we find 
      \begin{equation*}
        {\left( 1- \beta  \right)} \mathbb{ P } {\left( \exists _{ j \leq n } : x^{ t_0 } _{ t_j } \notin U_{ \epsilon  } ^{ 2 }  \right)} \leq \beta.
      \end{equation*}
      Therefore, \( \mathbb{ P } {\left( \exists_{ j \leq n } : x^{ t_0 } _{ t_j } \notin U^{ 2 } _{ \epsilon  }  \right)} \leq \tfrac{ \beta  }{ 1- \beta  }  \). Thus, Lemma~\ref{Largest step bound} applies, and we find
      \begin{equation}
        \mathbb{ P } {\left( \exists _{ j \leq k \leq n } : x^{ t_j } _{ t_{ k }  } \notin U^{ 4 } _{ \epsilon  }  \right)} \leq \tfrac{ \beta  }{ 1 - \beta  } . \label{Regularization of multiplicative processes (1)}
      \end{equation}
      Thus, we have shown that (\ref{Regularization of multiplicative processes (1)}) holds for all~\( \beta > 0 \) and~\( \epsilon \in ( 0 , \rho ') \) and suitable~\( H = H( \beta , \epsilon ) \). By applying (\ref{BCH formula I}) twice, there is \( \epsilon _0  \) such that for all~\( \epsilon \leq \epsilon _0 \), we have \( U_{ \epsilon  } ^{ 4 } \subseteq  U_{ \delta  }    \). Hence, (\ref{Regularization of multiplicative processes (1)}) holds for such~\( \epsilon  \) and~\( \beta = \tfrac{ \alpha  }{ 1+  \alpha  }  \), i.e.,~\( \alpha  = \tfrac{ \beta  }{ 1 - \beta  }  \). Then
      \begin{equation*}
        \mathbb{ P } {\left( \exists _{ j \leq k \leq n } : x^{ t_j } _{ t_{ k }  } \notin U _{ \delta   }  \right)} \leq \alpha .
      \end{equation*}
    \item\label{Regularization of multiplicative processes II} For all~\( \delta \in ( 0 , \rho ') \), we have \( \mathbb{ E } {\left( \bm{ O } _x^{ \delta  } ( \mathbb{ T } ) \right)} < \infty  \).
    \item [\textit{Proof.}] Let \( ( s _ m)_m \subseteq \mathbb{ T }  \) be an exhaustive sequence in~\( \mathbb{ T }  \) and define \( \mathbb{ T } _m = \{ s_j \}_{ j \leq m }  \). By Claim~\ref{Regularization of multiplicative processes I}, for all~\( \alpha > 0 \), \( \delta \in (0 , \rho ') \), there is \( H > 0 \) such that for all~\( n \in \mathbb{ N }  \) and all~\( t_0 \leq \dots \leq t_n \in \mathbb{ T } _m\) with~\( \abs{ t_n - t_0 } < H \), we have
      \begin{equation*}
        \mathbb{ P } {\left( \exists _{ j \leq k \leq n } : x^{ t_j } _{ t_k } \notin U_{ \delta  }  \right)} \leq \alpha .
      \end{equation*}
      Choose \( \alpha \leq 1/2 \), \( \delta \in (0 , \rho ' ) \). Then there is \( N \in \mathbb{ N } \) such that we can choose an adequate~\( H = \tfrac{ T }{ N }  \). Then, for each~\( \ell = 1 , \dots , N \), and~\( t^{ \ell } _0 \leq \dots \leq t^{ \ell } _n \in \mathbb{ T } _m \cap [T \tfrac{ \ell-1 }{ N }  , T \tfrac{ \ell }{ N }  ]\), we have
      \begin{equation*}
        \mathbb{ P } {\left( \exists _{ j \leq k \leq n  } : x^{ t^{ \ell } _j } _{ t^{ \ell } _k } \notin U_{ \delta  }  \right)} \leq \alpha  .
      \end{equation*}
      By Lemma~\ref{Expectation bound for oscillations},
      \begin{equation}
        \mathbb{ E } {\left( \bm{ O } _x^{ \delta  } ( \{ t_1^{ \ell } , \dots , t_n ^{ \ell } \} ) \right)} \leq \tfrac{ \alpha  }{ 1 - \alpha  } \leq 1. \label{Regularization of multiplicative processes (2)}
      \end{equation}
      By assumption, \( \limsup_{ m \to \infty  } \mathbb{ T } _m = \mathbb{ T }  \), hence \( \limsup_{ m \to \infty  } \mathbb{ T } _m \cap [T \tfrac{ \ell-1 }{ N }  , T \tfrac{ \ell }{ N }  ] = \mathbb{ T } \cap [ T \tfrac{ \ell-1 }{ N }  , T \tfrac{ \ell }{ N } ] \). Therefore, by Corollary~\ref{Properties of oscillation counter}, \( \lim_{ m \to \infty   } \bm{ O } _x ^{ \delta  } ( \mathbb{ T } _m \cap [T \tfrac{ \ell-1 }{ N }  , T \tfrac{ \ell }{ N }  ] ) = \bm{ O } _x ^{ \delta  } ( \mathbb{ T } \cap [T \tfrac{ \ell-1 }{ N }  , T \tfrac{ \ell }{ N }  ] ) \) and this sequence is monotone. Thus, the monotone convergence theorem applies and we have
      \begin{equation*}
        \lim_{ m \to \infty  } \mathbb{ E } {\left( \bm{ O } _x ^{ \delta  } ( \mathbb{ T } _m \cap [T \tfrac{ \ell-1 }{ N }  , T \tfrac{ \ell }{ N }  ] )  \right)} = \mathbb{ E } {\left( \bm{ O } _x ^{ \delta  } ( \mathbb{ T } \cap [T \tfrac{ \ell-1 }{ N } , T \tfrac{ \ell }{ N } ] ) \right)}.
      \end{equation*}
      By (\ref{Regularization of multiplicative processes (2)}), the limit is bounded by 1. In total, we obtain with Corollary~\ref{Properties of oscillation counter},
      \begin{equation*}
        \mathbb{ E } {\left( \bm{ O } _x ^{ \delta  } ( \mathbb{ T } \cap [0,T] ) \right)} \leq \sum_{ \ell=1 } ^{ N } \mathbb{ E } {\left( \bm{ O } _x^{ \delta  } ( \mathbb{ T } \cap [T \tfrac{ \ell-1 }{ N }  , T \tfrac{ \ell }{ N } ]  \right)} + N \leq 2N < \infty.
      \end{equation*}
      This proves~Claim~\ref{Regularization of multiplicative processes II}.
  \end{myclaims}
  In particular, the number of oscillations larger than~\( \delta  \) is almost surely finite. Let \( \Omega_n ' \) be the set of~\( \omega \in \Omega  \) for which the number of oscillations larger than~\( 1/n \) is finite. Then \( \mathbb{ P } {\left(  \Omega '_n \right)} =1\). By~\( \sigma  \)-additivity, the event~\( \Omega ' = \bigcap_{n=1}^{\infty } \Omega '_n \) is also almost sure. We may apply Lemma~\ref{Characterization of continuous functions} and use Claim~\ref{Regularization of multiplicative processes II} to define \( ( \overline{ x } _t ) _{ t }  \) by the rule
  \begin{equation*}
    \overline{ x }  _t {\left( \omega  \right)} =
      \begin{cases}
        \lim_{t_n  \searrow  t, \,   t_n \in \mathbb{ T }   } x _{ t_n }  ( \omega ), &  \omega \in \Omega ',\\
        e, & \omega \notin \Omega '
      \end{cases}
  \end{equation*}
  and \( \overline{ x }  \) by~\( \overline{ x } ^{ s } _t  = ( \overline{ x }  _s)^{ -1 } \overline{ x }_t\).

  To see that \( \overline{ x }  \) is a modification of~\( x \), consider first \( s<t \). For \( s_n, t_n \in \mathbb{ T }   \) with~\( s_n, t_n \searrow s,t \), we have \( x^{ s_n }_{ t_n } \to x^{ s } _t \) in probability. Thus, there is a subsequence that converges almost surely to~\( x^{ s } _t \). By definition, \( x^{ s_n } _{ t_n } \to \overline{ x } ^{ s } _t \) almost surely as well. Therefore, \( x^{ s } _t = \overline{ x } ^{ s } _t  \) almost surely. Second, if \( s=t \), then similarly \( x^{ s } _s = e  \) almost surely and \( \overline{ x } ^{ s } _s = e \) (surely). Thus, \( x^{ s } _s = \overline{ x } ^{ s } _s \) almost surely. 

  The enumerated assertions follow.
\end{proof}
C\`adl\`ag functions are determined by their values on a dense set of the domain. We can use this fact to show an even stronger notion of continuity for multiplicative processes. 
\begin{corollary}
  \label{Stochastic uniform continuity}
   Let \( x \) be a multiplicative process in~\( \mathscr{G}  \) with c\`adl\`ag paths and \( T > 0 \). Then, for all~\( \alpha > 0 \) and \( \delta \in ( 0, \rho ') \), there is \( H > 0 \) such that for all~\( r<u < T \) with~\( \abs{ u-r }<H  \), we have
  \begin{equation*}
    \mathbb{ P } ( \exists_{ r<s<t < u }:\,  x^{ s } _t \notin U_{ \delta  }   ) \leq \alpha .
  \end{equation*}
\end{corollary}
\begin{proof}
  By Claim~\ref{Regularization of multiplicative processes I} in the proof of Theorem~\ref{Regularization of multiplicative processes}, for~\( \alpha > 0 \) and \( \delta \in (0, \rho ' ) \), there is \( H > 0 \) such that for all~\( \{r=t_0 < \dots < t_n = r+H \} \),
  \begin{equation*}
    \mathbb{ P } (\exists_{ j \leq k \leq n }:\, x^{ t_j } _{ t_k } \notin U_{ \delta  }  ) \leq \alpha .
  \end{equation*}
  Choose a dense sequence~\( (t_m)_m \subseteq [r,r+H] \) and define \( \mathbb{ T } _n = \{t_1, \dots, t_n \} \). Then, using continuity of probability measures,
  \begin{equation*}
    \mathbb{ P } (\exists_{ n,m \in \mathbb{ N }  }:\,  x^{ t_n } _{ t_m } \notin U_{ \delta  }  ) = \lim_{ n \to \infty  } \mathbb{ P } (\exists_{ j,k \leq n }:\, x^{ t_j } _{ t_k }  \notin U_{ \delta } ) \leq \alpha .
  \end{equation*}
  As \( x \) is assumed to have c\`adl\`ag paths, \( (x^{ s } _t)_{ s \leq t \in [r,u] }  \) is determined by the values of~\( (x^{ t_n } _{ t_m } )_{ n,m \in \mathbb{ N }  } \).
\end{proof}
Another consequence of the c\`adl\`ag property is that the d\'ebut theorem~\cite{bass2010measurability} applies. Thus, for closed sets \( A \subseteq \mathscr{ G } \), the hitting times~\( \tau^{ A } = \inf \{ t > 0 :\ x^{ 0 } _t \in A \}\) and \( \sigma ^A = \inf \{ t > 0 : x^{ t- } _{ t } \in A \}\) are stopping times with respect to the filtration \( ( \mathscr{ F } ^{ 0 } _t)_{ t }  \).

\subsection{Markov property of multiplicative processes}
We want to formally establish \( (x_t)_t \) as a Markov process and conclude the strong Markov property. Therefore, we define the probability measures~\( (\mathbb{ P } _{ (s,y) } ) _{ 0 \leq s < \infty  , y \in \mathscr{G}  }   \) by
\begin{equation*}
  \mathbb{ P } _{ (s,y) } ( x_{ t_1 }  \in A_1 , \dots, x_{ t_n } \in A_n ) = \mathbb{ P } ( y x^{ s } _{ t_1 }  \in A_1 , \dots ,y x^{ s } _{ t_n } \in A_1 ) 
\end{equation*}
for each~\( n \in \mathbb{ N } \), \( t_1, \dots,t_n \geq s  \) and~\( A_1,\dots,A_n \subseteq \mathscr{G} \) Borel-measurable. Under \( \mathbb{ P } _{ (s,y) }  \), the process \( (y^{ -1 } x_{ s+t } )_{ t \geq 0 }   \) becomes a multiplicative process, and we can assume that \( (x_t)_{ t \geq s}  \) has c\`adl\`ag paths.

It is clear that a multiplicative process~\( (x_t)_t \) is a time-inhomogeneous Markov process in the sense of Chapter~1.1 of~\cite{gikhman2004theory}. The process~\( (x_t)_t \) is also c\`adl\`ag. Therefore, by Remark~1 in Chapter~1.4 of~\cite{gikhman2004theory}, the process~\( (x_t)_t \) is a strong Markov process if the map
\begin{equation*}
  [0,t) \times \mathscr{G} \to \mathbb{ R } , \, (s,y) \mapsto \mathbb{ E }_{ (s,y) }  {\left( f(x_t ) \right)} 
\end{equation*}
is right-continuous in~\( s \) and continuous in~\( y \) for any bounded Borel-measurable function \( f: \mathscr{ G } \to \mathbb{ R }  \). This property is indeed satisfied as \(\mathbb{ E } _{ (s,y) } {\left( f(x_t ) \right)} = \mathbb{ E } {\left( f ( y x^{ s } _t ) \right)} \) almost surely. Therefore, we have the following result by Theorem~6 in Chapter~1.4 of~\cite{gikhman2004theory}.:

\begin{lemma}
  \label{Strong Markov property}
  Every multiplicative process with c\`adl\`ag paths is a strong Markov process and for any bounded Borel-measurable map~\( f: \mathscr{G}^m \to \mathbb{ R }  \), \( s \geq 0 \), any finite stopping time~\( \tau \) such that \( \tau \geq s \) almost surely, and any \( \mathscr{ F }^{ s }  _{ \tau  }  \)-measurable random variables \( \eta_{ 1 } ,\dots, \eta_{ m }  \geq \tau  \), it holds
    \begin{equation*}
    \mathbb{ E } _{ (s,x) } {\left( f( x_{ \eta_{ 1 }   }, \dots, x_{ \eta_{ m }  } ) \mid \mathscr{ F }^s _{ \tau  }  \right)}  = \mathbb{ E } _{ (\tau , x_{ \tau  })  } {\left( f( x_{ \eta_{ 1 }   }, \dots, x_{ \eta_{ m }  }  ) \right)} \qquad \text{for all \( x \in \mathscr{ G }  \)} .
  \end{equation*}
\end{lemma}

\subsection{Jumps of a multiplicative process}
We define \( \mathfrak{ B } _{ \epsilon  } \) consisting of Borel-measurable sets~\( A \subseteq \mathscr{ G }  \) such that \( A \cap \exp ( B_{ \epsilon  } ) = \emptyset  \). Thus, elements in~\( \mathfrak{ B } _{ \epsilon  }  \) have a distance of at least~\( \epsilon  \) from the identity~\( e \). Further, we denote \(\mathfrak{ B } _{ 0 } = \bigcup_{ \epsilon > 0 } \mathfrak{ B } _{ \epsilon  }  \).

If \( (x_t)_t  \) is a multiplicative process and \( A \in \mathfrak{ B } _{ 0 }  \), the hitting times \( \tau_{ 0 }^{ A }  = 0 \) and
\begin{equation*}
  \tau _{ n+1 }^{ A } = \inf \{ t > \tau_{ n } ^{ A }  : x^{ t- } _t \in A \},
\end{equation*}
are well-defined stopping times with respect to the filtration~\( (\mathscr{ F } ^{ 0 } _t)_{ t }  \). Since \( (x_t)_t \) has c\`adl\`ag paths, on every interval~\( [0,t] \), there are almost surely only finitely many occurrences~\( x^{ s- } _s  \in A\),~\( s \leq t \). As a result, \( \lim_{ n \to \infty  } \tau_{ n }^{ A }  = \infty  \) almost surely. We define the jump measure~\( \nu _t (A) \) by the rule
\begin{equation*}
  \nu_{ t } (A) = \sum_{ n=1 } ^{ \infty  } \mathbb{ I } ( \tau^{ A } _{ n } \leq t ). 
\end{equation*}
\begin{lemma}
  For~\( A \in \mathfrak{ B } _0  \), the process~\( ( \nu_{ t } (A) )_t\) is a Poisson process, i.e., there is a function~\( \lambda : [0, \infty ) \to [0, \infty ) \) which is non-decreasing such that 
  \begin{enumerate}[1)]
    \item \( \nu_{ 0 } (A) = 0 \) almost surely,
    \item the random variables \( \nu_{ t_1 } (A) - \nu_{ t_0 } (A) , \dots, \nu_{ t_{ n }  } (A) - \nu_{ t_{ n-1 }  } (A) \) are independent for \( t_0 < \dots < t_n \),
    \item for~\( s < t \), it holds \( \nu_{ t } (A) - \nu_{ s } (A) \sim \operatorname{Pois}{\left( \lambda( t ) - \lambda( s )  \right)} \).
  \end{enumerate}
  If, in addition, \( (x_t)_t \) has stationary increments, i.e., the law of~\( x^{ s } _t \) only depends on~\( t-s \), then \( (\nu_{ t } (A))_t \) is a L\'evy process. Moreover, if \( A \subseteq \operatorname{dom} ( \operatorname{log} ) \), then the~\( \mathfrak{g} \)-valued process~\( (X_t)_t \) defined by 
  \begin{equation*}
    X_t = \sum_{ n=1 } ^{ \infty  } \operatorname{log}  (x_{ \tau ^{ A } _{ n } - } )^{ -1 }  x_{ \tau ^{ A } _{ n }  } ) \cdot \mathbb{ I } {\bigl( \tau ^{ A } _{ n } \leq t \bigr)} , 
  \end{equation*}
  is an additive process.
\end{lemma}
\begin{proof}
  Clearly, \( ( \nu _t(A))_t \) has piecewise constant paths because \( (x_t)_t \) has c\`adl\`ag paths. The jumps of~\( (\nu _t(A))_t \) have a height of~\( 1 \) because \( \tau_{ n } ^{ A } < \tau_{ n+1 }^A \). By Theorem~1 in Chapter 1.1.4 of~\cite{ito2013stochastic}, it is sufficient to show that \( ( \nu_{ t } (A) )_t \) is an additive process. We have 
  \begin{equation}
    \nu_{ t } (A) - \nu_{ s } (A) = \sum_{ n=1 } ^{ \infty  } \mathbb{ I } ( \tau ^{ A } _{ n } \leq t ) - \mathbb{ I } ( \tau ^{ A } _{ n } \leq s ) = \sum_{ n=1 } ^{ \infty  } \mathbb{ I } ( s < \tau ^{ A } _{ n } \leq t). \label{Jump process is additive}
  \end{equation}
  Define the random integer~\( N = \inf \{ k \in \mathbb{ N } : \tau^{ A } _{ k } > s \} \). We claim that for all~\( n \in \mathbb{ N }  \), it holds \( \{ s < \tau^{ A } _{ (N-1)+n } \leq t \} \in \mathscr{ F } ^{ s } _t \). To see this, we define~\( ( \tilde{ x } ^{ r } _t ) _{ r,t } := (x^{ s+r } _{ s+t })_{ r,t }   \). This process is multiplicative with respect to~\( ( \tilde{ \mathscr{ F }  } ^{ r } _t)_{ r,t } := ( \mathscr{ F } ^{ s+r } _{ s+t } )_{ r,t }  \). Thus, the hitting times~\( \tilde{ \tau  } _{ 0 } := 0 \) and \( \tilde{ \tau  }^{ A }  _{ n+1 } := \inf \{ t > \tilde{ \tau  }^{ A }  _n : \tilde{ x } ^{ t- } _t \in A \} \) are stopping times with respect to \( ( \tilde{ \mathscr{ F }  } ^{ 0 } _t)_{ t } \), i.e., \( \{ 0 < \tilde{ \tau  }^{ A }  _n \leq t-s \} \in \tilde{ \mathscr{ F }  } ^{ 0 } _{ t-s }  \), implying
  \begin{equation*}
    \{ s < \tau_{ (N-1)+n }^{ A }  \leq t \} \in \mathscr{ F } ^{ s } _t,
  \end{equation*}
  because \( \{ s < \tau_{ (N-1)+n } ^{ A } \leq t \} = \{ 0 < \tilde{ \tau  }^{ A }  _n \leq t-r \} \) and \( \tilde{ \mathscr{ F }  } ^{ 0 } _{ t-s } = \mathscr{ F } ^{ s } _t \). 

  We continue showing that \( (\nu _t(A) )_t \) is additive. Using (\ref{Jump process is additive}) and the fact that \( \mathbb{ I } ( s < \tau ^{ A } _{ n } \leq t) = 0 \) if \( n<N \) , we obtain
  \begin{equation*}
    \nu_{ t } (A) - \nu_{ s } (A) = \sum_{ n=1 } ^{ \infty  } \mathbb{ I } ( s < \tau ^{ A } _{ n } \leq t) = \sum_{ n=N } ^{ \infty  } \mathbb{ I } ( s < \tau ^{ A } _{ n } \leq t)  = \sum_{ n=1 } ^{ \infty  } \mathbb{ I } ( s < \tau ^{ A } _{ (N-1) +n } \leq t) .
  \end{equation*}
  Hence, the random variable~\( \nu _t (A) - \nu_{ s } (A) \) is \( \mathscr{ F } ^{ s } _t \)-measurable. By definition, \( \nu_{ 0 } (A) = 0\). It is left to show that \( ( \nu_{ t } (A))_t  \) is stochastically continuous. For~\( t> 0 \), we have
  \begin{align*}
    &\lim_{ H \searrow 0 } \mathbb{ P } {\left( \forall _{ s \in [0,t] , \abs{ h } < H}  :\, \abs{\nu_{ s+h } (A) - \nu_{ s  } (A) } > 0 \right)} \\
    & \quad = \lim_{ H \searrow  0 } \mathbb{ P } {\left( \forall _{ s \in [0,t], \abs{ h } < H } :\, (x_{ s })^{ -1 } x_{ s+h } \in  A  \right)}.
  \end{align*}
  By Corollary~\ref{Stochastic uniform continuity}, the right-hand side is zero. The assertion about~\( (X_t)_t \) follows.
\end{proof}

\section{Results for Multiplicative Processes with Bounded Jumps}\label{Results for multiplicative processes with bounded jumps} 
For additive processes and in particular L\'evy processes~\( (L_t)_t \) on~\( \mathbb{ R }  \), we know that if the jumps~\( \abs{ L_{ t } - L_{ t- }  }  \) are bounded, then all moments of~\( L_t \) are finite (see Theorem~2.4.7 of~\cite{applebaum2009levy}), and even \( \mathrm{e}^{ L_t }  \) has a finite expectation. We would like to state an assertion for multiplicative processes that resembles this fact. There are some problems however: A Lie group is, in general, not endowed with a global norm to generalize the absolute value, and we cannot even consider an expectation~\( \mathbb{ E } ( x^{ s } _t ) \) as \( \mathscr{G}  \) does not carry an addition operation. 

However, we can use the notion considered in Lemma~\ref{Local metric of Lie groups}. We can represent every~\( g \in \mathscr{G}  \) in the connected component containing~\( e \) by any given neighborhood~\( B_{ \delta  } \) around zero as
\begin{equation}
  g = \operatorname{exp}( G_1 ) \cdots \operatorname{exp}( G_m ) \label{step counter}
\end{equation}
for some~\( G_1, \dots, G_m \in B _{ \delta  } \) by the following well-known assertion mentioned on page~63 of~\cite{varadarajan2013lie}:
\begin{lemma}
  \label{Cover of the connected component}
  Let \( \mathscr{G}  \) be a connected Banach-Lie group. Then, for any~\( g \in \mathscr{G}  \), and any open set~\( V \subseteq \mathscr{G}  \) around the identity, there is \( m \in \mathbb{ N }  \) and \( g_1, \dots ,g_m \in V\) such that \( g = g_1 \cdots g_m \).
\end{lemma}
One can interpret the minimal such~\( m \) as some form of distance. Indeed, this quantity satisfies a triangle inequality as for~\( g,h \in \mathscr{ G }  \),
\begin{equation}
  \inf \{ m: \, gh \in V^{ m } \} \leq \inf \{ m :\, g \in V^{ m } \} + \inf \{ m :\, h \in V^{ m } \}, \label{triangle inequality for steps}
\end{equation}
for all open sets~\( V \) by the same argument as in the proof of Lemma~\ref{Local metric of Lie groups}.

In our case, we do not have to assume that \( \mathscr{G}  \) is connected to guarantee that the process remains in the connected component with~\( e \). It suffices to assume that \( (x^{ s }_t)_{ s,t } \) only has jumps within the domain of the logarithm. Then it does not jump out of the connected component of~\( \mathscr{G}  \). 

To generalize the expectation bound for L\'evy processes, we need the notion of bounded jumps. A c\`adl\`ag process~\( x = (x_t)_{ t }  \) is said to have \emph{jumps bounded in~\( A \)} if \( A \subseteq \mathscr{G}  \) contains the identity and
\begin{equation*}
  \mathbb{ P } {\left( \forall t \geq 0 :\, x^{ t- } _t \in A \right)} =1.
\end{equation*}
The process~\( x \) is said to have \emph{bounded jumps} if the jumps of~\( x \) are bounded in a set~\( (U_{ \delta  } )^{ n } \) for some~\( \delta > 0 \) such that \( U_{ \delta  } \subseteq  \operatorname{dom} (\operatorname{log} ) \) and some~\( n \in \mathbb{N }  \). Now we are in a position to state a bound for the distance as described earlier. 
The proof contains an argument reminiscent of Lemma~2 in Chapter~IV.1 of~\cite{gikhman2004theory}.
\begin{theorem}
  \label{Finite expectation of distance}
  Let \( x = (x^{ s } _t ) _{ s,t }  \) be a multiplicative process in~\( \mathscr{G}  \) with c\`adl\`ag paths. Suppose \( x \) has bounded jumps. Then, for all~\( 0< r < u \), \( \alpha \in \mathbb{ R }  \) and \( V \subseteq \mathscr{G} \) open and containing~\( e \in \mathscr{G}  \),
  \begin{equation}
    \mathbb{ E } {\left( \sup_{ r \leq s \leq t \leq u } \mathrm{e} ^{ \alpha   \inf \left\{ m:\, x^{ s } _t \in V^{ m }  \right\} } \right)} < \infty. \label{Finite expectation of distance (1)}
  \end{equation}
\end{theorem}
A trivial but helpful fact to estimate the integral of a decreasing function~\( g: [0,\infty ) \to [0,\infty ) \) using its values along a sequence~\( ( \xi_n ) _n \subseteq [0, \infty )\) such that \( 0 = \xi_0 < \xi_1 < \dots < \xi_n \to \infty  \) is given by
\begin{equation}
  \int_{ 0 } ^{ \infty  } g(\xi) \, \dd \xi \leq \sum_{ n=0 } ^{ \infty  } g(\xi_n) \cdot ( \xi_{ n+1 } - \xi_n  ). \label{Integral estimation for a product}
\end{equation}
\begin{proof} 
  We first show a weaker result: We claim that
  \begin{equation}
    \sup_{ r \leq s \leq u } \mathbb{ E } {\left( \sup_{ s \leq t \leq u} \mathrm{e} ^{ \alpha M(x^{ s } _t) } \right)} < \infty, \label{Finite expectation of distance (2)}
  \end{equation}
  where \( M(x) = \inf \left\{ m:\, x \in V^{ m }  \right\} \). Notice that \( \sup_{ r \leq s \leq t \leq u } \mathrm{e} ^{ \alpha M(x^{ s } _t) } \) is measurable by the measurable projection theorem. Let \( x \) have jumps bounded in~\( (U_{ \delta  } )^{ N }  \) for some \( N \in \mathbb{ N }  \). By (\ref{Symmetric system}), there is \( \epsilon > 0 \) such that \( \exp ( B_{ \epsilon  } ) \subseteq V \) because \( e \in V \). Hence, we may assume \( V = \exp ( B_{ \epsilon  } ) \) as that makes \( V \) smaller and \( M \) larger. By boundedness of the jumps, there is \( j \in \mathbb{ N }  \) such that~\( ( U_{ \delta  } )^{ N }   \subseteq V^{ j }  \). We can safely assume that {\( j \geq 3 \)}.

  Notice that for all~\( s < t  \), all stopping times~\( \tau  \), and all~\( \omega \in \{ \omega \in \Omega :\tau ( \omega ) \in [s,t] \} \),
  \begin{equation}
    M(x^{ s } _t( \omega )) \leq M(x ^{ s } _{ \tau - } ( \omega ) ) + M( x^{ \tau -  } _{ \tau  }( \omega ) ) + M( x^{ \tau  } _{ t } ( \omega )) \label{Finite expectation of distance (3)}
  \end{equation}
  by (\ref{triangle inequality for steps}). Also, \( M( x^{ \tau  - } _{ \tau   } ) \leq j \) because \( x^{ \tau - }_{ \tau  } \in (U_{ \delta  } )^{ N } \subseteq V^{ j }  \). In particular, for~\( \tau = \tau _s := \inf \left\{ z>s:\, x^{ s } _z \notin V \right\}  \), we find {\(  M(x^{ s } _{ \tau_s-  } )  \leq 2\)} by definition of~\( \tau_{ s }  \). Hence,
  \begin{equation}
    M ( x^{ s } _t ( \omega ) ) \leq {2}+j +  M(x^{ \tau _{ s }  } _{ t   } ( \omega ) ) \qquad \text{for all~\( \omega \in \{ \tau_{ s }  \leq u \} \)} . \label{Finite expectation of distance (4)}
  \end{equation}
  By stochastic uniform continuity (Corollary~\ref{Stochastic uniform continuity}), for every~\( \beta > 0 \) and \( T>0 \), there is \( H > 0 \) such that for all \( r \in [0, T] \)
  \begin{equation}
    \sup_{ r \leq  s \leq  r + H}  \mathbb{ P } ( \tau_{ s  }  \leq  r+ H ) < \mathrm{e} ^{ - \beta  } \label{Finite expectation of distance (5)},
  \end{equation}
  because \(  \sup_{ r \leq  s \leq  r + H } \mathbb{ P } ( \tau_{ s  } \leq  r+ H ) = \sup_{ r \leq  s \leq  r + H } \mathbb{ P } ( \exists_{ t \in [s,r+H] }: x^{ s } _t \notin V ) \) as the events are identical. Choose some~\( \beta > 4 \alpha j \) and \( H > 0 \) such that (\ref{Finite expectation of distance (5)}) holds. By taking a smaller value for \( H \), we can assume \( T/H =: K \in \mathbb{ N }  \)
  \begin{mycases}
    \item \( u \leq r+H \). \label{Finite expectation of distance I}
    \item[] Let \( \gamma \geq  \mathrm{e} ^{ 2 \alpha  }  \). Then \( \sup_{ s \leq  t \leq  u } \mathrm{e} ^{ \alpha M( x^{ s } _t )  } > \gamma  \) only if \( \tau_{ s } \leq  u \) and in that case 
      \begin{equation*}
        \sup_{ s \leq  t \leq  u } \mathrm{e} ^{ \alpha M ( x^{ s } _t ) }  =  \sup_{   \tau_{ s }   \leq  t \leq  u}  \mathrm{e} ^{ \alpha M( x^{ s } _t ) }.
      \end{equation*}
      We find therefore
      \begin{align}
        &\mathbb{ P } {\left( \sup_{ s \leq t \leq u } \mathrm{e}^{ \alpha  M(x^{ s } _t ) } > \gamma  \right)} \nonumber\\
        &\qquad = \mathbb{ P } {\left( \sup_{  \tau_{ s }  \leq t \leq u  } \mathrm{e}^{ \alpha M ( x^{ s } _t ) }> \gamma  , \, \tau_{ s } \leq u \right)} \nonumber \\
        &\qquad \overset{\mathclap{(\ref{Finite expectation of distance (4)})}}{\leq} \, \mathbb{ P } {\left( \sup_{ \tau_{ s } \leq t \leq u } \mathrm{e}^{ \alpha (j+{2}) } \mathrm{e}^{ \alpha M(x^{ \tau _s } _t ) } > \gamma , \, \tau _s \leq u \right)} \nonumber\\
        &\qquad \overset{\text{(i)}}{\leq} \int_{ s } ^{ u } \mathbb{ P } {\left( \sup_{ z \leq t \leq u } \mathrm{e}^{ \alpha (j+ {3}) }\mathrm{e}^{ \alpha  M(x^{ z } _t ) } > \gamma \right)} \cdot \mathbb{ P } {\left( \tau_{ s } \in \dd z \right)}  \nonumber\\
        & \qquad \leq \sup_{ r \leq s \leq u } \mathbb{ P } {\left( \sup_{ s \leq t \leq u } \mathrm{e}^{ \alpha (j+{3}) } \mathrm{e}^{ \alpha M (x^{ s } _t ) } > \gamma  \right)} \cdot \sup_{ r \leq s \leq u } \mathbb{ P } {\left( \tau_{ s } \leq u \right)}.\label{Finite expectation of distance (6)}
      \end{align}
    We show that (i) holds. Notice that
    \begin{align}
      &\mathbb{ P }  {\left( \sup_{ \tau_{ s } \leq t \leq u } \mathrm{e} ^{ \alpha (j+{2}) } \mathrm{e} ^{ \alpha M(x^{ \tau_{ s }  } _t) }  > \gamma  ,  \tau_{ s } \leq u   \right)}  \nonumber \\ 
      & \quad = \mathbb{ E } {\left( \mathbb{ E } {\left( \mathbb{ I }  \Bigl\{  \sup_{ \tau_{ s } \leq t \leq u}  \mathrm{e} ^{ \alpha (j+ {2}) } \mathrm{e} ^{ \alpha M(x^{ \tau_{ s }  } _t) }  > \gamma \Bigr\}  \cdot \mathbb{ I } \{ \tau_{ s } \leq u \}  \middle\vert \mathscr{ F } _{ \tau_{ s }  }  \right)}  \right)} \nonumber \\ 
      & \quad = \mathbb{ E } {\left(   \mathbb{ I } \{ \tau_{ s }  \leq u \} \mathbb{ E } {\left( \mathbb{ I }  \Bigl\{  \sup_{ \tau_{ s } \leq t \leq u} \mathrm{e} ^{ \alpha (j+ {2}) } \mathrm{e} ^{ \alpha M(x^{ \tau_{ s }  } _t) }  > \gamma \Bigr\}  \middle\vert \mathscr{ F } _{ \tau_{ s }  }  \right)}  \right)}.\label{Finite expectation of distance (7-1)} 
    \end{align}
    By Lemma \ref{Strong Markov property} for all \( n \in \mathbb{ N }  \), we have that 
    \begin{align}
        &\mathbb{ E } {\left( \mathbb{ I } \biggl\{ \sup_{ k \leq 2^n } \mathrm{e} ^{ \alpha (j+ {2}) } \mathrm{e} ^{  \alpha  M \bigl( x^{ \tau_{ s }  } _{ \tau_{ s } + \frac{ k }{ 2^n } (u- \tau_{ s } ) } \bigr) } > \gamma  \biggr\}  \middle\vert \mathscr{ F } _{ \tau_{ s }  }  \right)} \nonumber \\
        & \quad = \mathbb{ E }_{ ( \tau_{ s } , x_{ \tau_{ s }  } ) }  {\left(    \mathbb{ I } \bigg\{ \sup_{ k \leq 2^n } \mathrm{e} ^{ \alpha (j+ {2}) } \mathrm{e} ^{  \alpha  M\bigl(x^{ \tau_{ s }  } _{ \tau_{ s } + \frac{ k }{ 2^n } (u- \tau_{ s } ) } \bigr)} > \gamma  \biggr\}   \right)} \nonumber .
    \end{align}
    For \( (g_n)_n \subseteq \mathscr{ G }  \) with \( \lim_{ n \to \infty } g_n = g \), we have \( \liminf_{ n \to \infty  } M(g_n) \geq M(g) -1 \), because \( M \) has jumps of size \( 1 \) at points of discontinuity. Hence, 
    \begin{equation}
      \liminf_{ n \to \infty  } \sup_{ k \leq 2^{ n }  } M(x^{ \tau_{ s }  } _{ \tau_{ s } + \frac{ k }{ 2^{ n }  } (u -\tau_{ s } ) } ) + 1 \geq \sup_{ \tau_s ( \omega ) \leq t \leq u } M(x^{ \tau_{ s }  } _t) .
    \end{equation} Therefore,
    \begin{align}
      &  \mathbb{ E } {\left(   \mathbb{ I } \{ \tau_{ s } \leq u \} \mathbb{ E } {\left( \mathbb{ I }  \biggl\{  \sup_{ \tau_{ s } \leq t \leq u } \mathrm{e} ^{ \alpha (j+ {2}) } \mathrm{e} ^{ \alpha M(x^{ \tau_{ s }  } _t) }  > \gamma \biggr\}  \middle\vert \mathscr{ F } _{ \tau_{ s }  }  \right)}  \right)} \nonumber \\
      & \quad \leq \mathbb{ E } {\left(   \mathbb{ I } \{ \tau_{ s } \leq u \} \mathbb{ E } {\left( \lim_{ n \to \infty  } \mathbb{ I }  \biggl\{  \sup_{ k \leq 2^{ n }  } \mathrm{e} ^{ \alpha (j+ {3}) } \mathrm{e} ^{ \alpha M \bigl( x^{ \tau_{ s } } _{ \tau_{ s } + \frac{ k }{ 2^{ n }  } ( u- \tau_{ s } ) }      \bigr)  }  > \gamma \biggr\}  \middle\vert \mathscr{ F } _{ \tau_{ s }  }  \right)}  \right)} \nonumber \\
      & \quad \overset{\mathclap{\text{(ii)}}}{=}  \lim_{ n \to \infty  } \mathbb{ E } {\left(   \mathbb{ I } \{ \tau_{ s } \leq u \} \mathbb{ E } {\left( \mathbb{ I }  \biggl\{  \sup_{ k \leq 2^{ n }  } \mathrm{e} ^{ \alpha (j+ {3}) } \mathrm{e} ^{ \alpha M \bigl(x^{ \tau_{ s } } _{ \tau_{ s } + \frac{ k }{ 2^{ n }  } ( u- \tau_{ s } ) } \bigr) }  > \gamma \biggr\}  \middle\vert \mathscr{ F } _{ \tau_{ s }  }  \right)}  \right)} \nonumber \\
      & \quad =  \lim_{ n \to \infty  } \mathbb{ E } {\left(   \mathbb{ I } \{ \tau_{ s } \leq u \} \mathbb{ E }_{ ( \tau_{ s } , x_{ \tau_{ s }  } ) }  {\left( \mathbb{ I }  \biggl\{  \sup_{ k \leq 2^{ n }  } \mathrm{e} ^{ \alpha (j+ {3}) } \mathrm{e} ^{ \alpha M \bigl(x^{ \tau_{ s } } _{ \tau_{ s } + \frac{ k }{ 2^{ n }  } ( u- \tau_{ s } ) } \bigr) }  > \gamma \biggr\}  \right)}  \right)} \nonumber \\
      &  \quad \overset{\mathclap{\text{(iii)}}}{\leq} \mathbb{ E } {\left(   \mathbb{ I } \{ \tau_{ s } \leq u \} \mathbb{ E }_{ ( \tau_{ s }, x_{ \tau_{ s }  } )  }  {\left( \mathbb{ I }  \biggl\{  \sup_{ \tau_{ s } \leq t \leq u } \mathrm{e} ^{ \alpha (j+ {3}) } \mathrm{e} ^{ \alpha M(x^{ \tau_{ s }  } _t) }  > \gamma \biggr\}   \right)}  \right)}.
    \end{align}
    Here, we applied monotone convergence in (ii) and used that 
    \begin{equation*}
      \sup_{ \tau_{ s } \leq t \leq u }\mathrm{e} ^{ \alpha M\bigl(x^{ \tau_{ s }  } _t\bigr) }  \geq     \sup_{ k \leq 2^{ n }  }  \mathrm{e} ^{ \alpha M\bigl(x^{ \tau_{ s } } _{ \tau_{ s } + \frac{ k }{ 2^{ n }  } ( u- \tau_{ s } ) } \bigr) } 
    \end{equation*}
    in (iii). Therefore, by (\ref{Finite expectation of distance (7-1)})
    \begin{align}
      &  \mathbb{ P } {\left( \sup_{ \tau_{ s } \leq t \leq u } \mathrm{e}^{ \alpha (j+ {2}) } \mathrm{e}^{ \alpha M(x^{ \tau _s } _t ) } > \gamma , \, \tau _s \leq u \right)} \nonumber \\
      & \quad \leq  \mathbb{ E } {\left( \mathbb{ I }  \{ \tau_{ s } \leq  u \}  \mathbb{ P }_{ ( \tau_{ s }  , x_{ \tau_{ s } } ) }  {\left( \sup_{ \tau_{ s }  \leq t  \leq  u }  \mathrm{e}^{ \alpha (j+ {3}) }\mathrm{e}^{ \alpha  M(x^{ \tau_{ s }  } _t)  } > \gamma    \right)} \right)}  \nonumber \\
       & \quad = \int_{ [s,u] \times \mathscr{ G }  } \mathbb{ P } _{ (z,y) }   {\left(  \sup_{ z \leq  t \leq  u }    \mathrm{e}^{ \alpha (j+ {3}) } \mathrm{e}^{ \alpha M(x^{ z } _t ) }    > \gamma    \right)}\cdot \mathbb{ P } {\left( \tau_{ s } \in \dd z, x_{ z } \in \dd y \right)} \nonumber  \\
       & \quad \overset{\mathclap{\text{(iv)}}}{=} \int_{ [s,u] \times \mathscr{ G } } \mathbb{ P }  {\left(  \sup_{ z \leq  t \leq  u }    \mathrm{e}^{ \alpha (j+ {3}) } \mathrm{e}^{ \alpha M(x^{ z } _t ) }    > \gamma    \right)}\cdot \mathbb{ P } {\left( \tau_{ s } \in \dd z, x_{ z } \in \dd y \right)} \nonumber  \\
       & \quad = \int_{ s } ^{ u } \mathbb{ P }   {\left(  \sup_{ z \leq  t \leq  u }    \mathrm{e}^{ \alpha (j+ {3}) } \mathrm{e}^{ \alpha M(x^{ z } _t ) }    > \gamma    \right)}\cdot \mathbb{ P }  {\left( \tau_{ s } \in \dd z \right)} ,
    \end{align}
    where in (iv) we used the independence of the increments.
    Thus, (i) holds.

  Applying (\ref{Finite expectation of distance (6)}) recursively, we obtain for all~\( \gamma \geq e^{ 2 \alpha  }  \)
      \begin{equation} 
        \sup_{ r \leq s \leq u } \mathbb{ P } {\left( \sup_{ s \leq t \leq u } \mathrm{e}^{ \alpha M(x^{ s } _t ) } > \gamma \mathrm{e}^{ \alpha k(j+ {3}) }  \right)} \leq \sup_{ r \leq s \leq u} \mathbb{ P } {\left( \tau_{ s } \leq u \right)}^{ k } .\label{Finite expectation of distance (8)}
      \end{equation}
      We find
      \begin{align}
          &\mathbb{ E } {\left( \sup_{ s \leq t \leq u } \mathrm{e}^{ \alpha M( x^{ s } _t )  } \right)} \nonumber\\
          & \quad \leq  \mathrm{e} ^{ 2 \alpha  }  + \int_{ \mathrm{e} ^{ 2\alpha  }  } ^{ \infty  } \mathbb{ P } {\left( \sup_{ s \leq t \leq u } \mathrm{e}^{ \alpha  M( x^{ s } _t )   } > \gamma \right)} \dd \gamma \nonumber\\
          & \quad \overset{\mathclap{(\ref{Integral estimation for a product})}}{\leq} \, \mathrm{e} ^{ 2 \alpha  } +  \sum_{ k=0 } ^{ \infty  } \mathbb{ P } {\left( \sup_{ s \leq t \leq u} \mathrm{e}^{ \alpha  M(x^{ s } _t ) } > \mathrm{e} ^{ 2 \alpha  } \mathrm{e} ^{ \alpha k (j+ {3}) }  \right)} \cdot   \mathrm{e} ^{ 2 \alpha  } \mathrm{e} ^{ \alpha (k+1) (j+2) }   \nonumber\\
          & \quad \overset{\mathclap{(\ref{Finite expectation of distance (8)})}}{\leq}  \, \,  \mathrm{e} ^{ 2 \alpha  } +  \sum_{ k=0 } ^{ \infty  } \sup_{ r \leq  s \leq  u} \mathbb{ P } {\left( \tau_{ s } \leq u \right)} ^{ k } \cdot   \mathrm{e} ^{ 2 \alpha  } \mathrm{e} ^{ \alpha (k+1) (j+ {3}) } \nonumber\\
          & \quad \overset{\mathclap{\text{(v)}}}{\leq} \mathrm{e} ^{ 2 \alpha  } + \sum_{ k=0 } ^{ \infty  } \mathrm{e} ^{ - \beta k }  \mathrm{e} ^{ 2 \alpha  } \mathrm{e} ^{ 4 \alpha kj }, \label{Finite expectation of distance (9)}
      \end{align}
      where we use (\ref{Integral estimation for a product}) with~\( \xi _k = \mathrm{e} ^{ 2 \alpha  } \mathrm{e} ^{ \alpha k (j+2) }  \). In (v), we use the estimate~\( (k+1)(j+ {3}) \leq 2k \cdot 2j \) as {\( j \geq 3 \)} and (\ref{Finite expectation of distance (5)}). 
  
  \item \( u > r+H \).
  \item[] We find that~\( u < T \leq r+KH \). Then, for all~\( s \in [r,r+KH] \),
      \begin{align*}
          &\sup_{ s \leq t \leq u } \mathrm{e} ^{ \alpha M (x^{ s } _t ) } \\
          & \quad \leq \quad  \sup_{ \mathclap{ s \leq t_1 \leq s+H} } \quad   \mathrm{e}^{ \alpha M(x^{ s } _{ t_1 }  ) }  \times \cdots \times \qquad  \sup_{ \mathclap{ s+(K-1)H \leq t_K \leq s+KH} } \qquad \mathrm{e}^{ \alpha \inf M(x^{ s+ (K-1)H } _{ t_K } )}  .
      \end{align*}
      Now the expectation of this product is the product of the expectations of these factors by independence. By Case \ref{Finite expectation of distance I}, these are finite.
  \end{mycases}
  This proves (\ref{Finite expectation of distance (2)}). Now we are in a position to show (\ref{Finite expectation of distance (1)}). Notice that
  \begin{align}
    M( x^{ s } _t ) & =  M(x^{ s }_u (x^{ t } _u ) ^{ -1 } )   \leq M (x^{ s }_u ) + M(x^{ t } _u ), \label{Finite expectation of distance (10)}
  \end{align}
  since \( (x^{ t } _u)^{ -1 } \in V \) iff \( x^{ t } _u \in V \) by assumption on \( V \).
  \begin{align*}
    \mathbb{ E } {\left( \sup_{ r \leq s \leq t \leq u } \mathrm{e}^{  \alpha M (x^{ s } _t ) } \right)} & \leq \mathbb{ E } {\left( \sup_{ r \leq s \leq t \leq u } \mathrm{e}^{ \alpha {\left( M(x^{ r } _s ) + M(x^{ r } _t )\right)}} \right)}  \\
                                                                                                                                                         & \leq \mathbb{ E } {\left( \sup_{ r \leq t \leq u } \mathrm{e}^{2 \alpha {\left( M(x^{ r } _t )\right)}} \right)}  \\
                                                                                                                                                         & \leq \sup_{ r \leq s \leq u} \mathbb{ E } {\left( \sup_{ s \leq t \leq u}  \mathrm{e}^{  2\alpha  M( x^{ s } _t ) } \right)} .                           
  \end{align*}
  By (\ref{Finite expectation of distance (2)}), the right-hand side is finite. This proves the theorem.
\end{proof}
Lie groups, in general, do not have a global metric, but connected Lie groups do. If a process has bounded jumps, then this process remains within the connected component containing the identity. Therefore, it is often not a loss of generality to assume that the considered Lie group is connected. In this case, the Birkhoff-Kakutani theorem ensures the existence of a left-invariant metric as well as a right-invariant metric~\cite{montgomery1955topological}. Since a metric~\( d \) on a Banach space can only induce a norm if \( d(u+w,v+w)=d(u,v) \) for all vectors~\( u,v,w \), metrics that are left- or right-invariant on a Lie group behave similarly to norms in that aspect.
\begin{theorem}\label{Finite expectation of metric}
  Suppose \( d \) is a left-invariant metric on~\( \mathscr{ G }  \) and \( (x_t)_t  \) is a multiplicative process with c\`adl\`ag paths and bounded jumps. Then, for all~\( 0 <r < u  \) and \( \alpha \in \mathbb{ R }  \), 
\begin{equation}
  \mathbb{ E } {\left( \sup_{ r  \leq  s  \leq  t \leq  u } \mathrm{e} ^{ \alpha d ( x_s, x_t  ) } \right)} < \infty . \label{Finite expectation of metric (1)}
\end{equation}
Moreover, for fixed~\( T > 0 \), we have~\( \lim_{ \abs{ u-r } \to 0 }  \mathbb{ E } {\left( \sup_{  r \leq  s \leq  t \leq  u } \mathrm{e} ^{ \alpha d ( x_s, x_t  ) } -1 \right)} = 0 \), where \( r < u < T \).
\end{theorem}

The theorem is shown analogously and we use the left-invariance of the metric together with the assumption that the process is multiplicative from left to right, i.e., \( x^{ s } _t x^{ t } _u = x^{ s } _u \). But the same result as in Theorem \ref{Finite expectation of metric} holds if \( d  \) is a right-invariant metric instead of a left-invariant one. This can be seen noting that if \( (x^{ s } _t)_{ s,t }  \) has bounded jumps, then so does \( ( (x^{ s } _t)^{ -1 } )_{ s,t }  \) which is multiplicative from right to left.

\section{Multiplicative Processes on the Heisenberg Group}\label{Multiplicative processes on the Heisenberg group} 
Following the notation of~\cite{eberlein1994geometry}, a \( (2n+1) \)-dimensional vector space \( \mathfrak{ h }  \) with basis \( \{ x_1,y_1, \dots,\allowbreak x_n, y_n, z \} \) is called a \emph{Heisenberg algebra} if it is endowed with a Lie bracket satisfying the \emph{Heisenberg relations}
\begin{alignat*}{2}
  &[x_i, y_i ] = z \qquad && \text{for \(i \leq n\)},\\
  &[x_i, z ] = 0,  && \text{for \( i \leq n \)}, \\
  &[x_i,x_j] = [y_i,y_j] = [x_i,y_j] = 0 \qquad  && \text{for \( i \neq j \leq n \).} 
\end{alignat*}
We construct an infinite-dimensional analogue as follows: Assume \( p \in (1,\infty ) \) and \( q = \tfrac{ p }{ p-1 }  \). As is customary, \( \ell^p( \mathbb{ N } ) \) is defined as the space of~\( p \)-integrable~\( \mathbb{ R }   \)-valued sequences. Since \( \tfrac{ 1 }{ p } + \tfrac{ 1 }{ q  } = 1 \), the space~\( \ell^q( \mathbb{ N } ) \) is the dual of~\( \ell^p(\mathbb{ N })  \) and vice-versa. We denote by~\( {\left\langle x \vert y \right\rangle}   \) the dual pairing of elements~\( x \in \ell^p( \mathbb{ N } ) \) and \( y \in \ell^q ( \mathbb{ N } ) \). We define the Lie group~\( \mathscr{ G } = \ell^p ( \mathbb{ N } ) \times \ell^q ( \mathbb{ N } ) \times \mathbb{ R } \) with multiplication by the rule 
\begin{equation*}
  {\left( x_1 , y_1 , z_1 \right)} \cdot  {\left( x_2 , y_2 , z_2 \right)} = {\left( x_1 + x_1 , y_1 + y_2 , z_1 + z_2 + \tfrac{ 1 }{ 2 } (\langle x_1 \vert y_2 \rangle  - \langle x_2 \vert y_1 \rangle) \right)}.
\end{equation*}
This multiplication map is indeed smooth. Each element~\( (x,y,z) \in \mathscr{ G }  \) has an inverse~\( (-x,-y,-z) \), and this mapping is smooth as well. Hence, \( \mathscr{ G }  \) becomes a Banach-Lie group on the manifold~\( \ell^p ( \mathbb{ N } ) \times \ell^q ( \mathbb{ N } ) \times \mathbb{ R } \). It can be shown that the associated Lie algebra is \( \mathfrak{g} = \ell^p ( \mathbb{ N } ) \times \ell^q ( \mathbb{ N } ) \times \mathbb{ R } \) with the Lie bracket given by \( [(x_1,y_1,z_1),(x_2,y_2,z_2)] = (0,0,  \langle x_1 \vert y_2 \rangle - \langle x_2 \vert y_1 \rangle ) \). We denote by \( (e_j)_j \) the canonical basis of \( \ell^{ p } ( \mathbb{ N } ) \) and by \( (f_j)_j \) the canonical basis of \( \ell^{q} ( \mathbb{ N } ) \). The collection~\( \{x_1,y_1, x_2,y_2, \dots, z \}\), where \( x_j = (e_j,0,0) \), \( y_k = (0,f_j,0)\) and \( z = (0,0,1) \) forms a basis of~\( \mathfrak{g} \). It can be checked that this collection also satisfies the Heisenberg relations. In this sense the Lie group \( \mathscr{ G }  \) deserves the name Heisenberg group. Similarly defined Heisenberg groups are studied in~\cite{kac1990infinite} and further considered in a stochastic context in~\cite{driver2008heat}.

We take an additive filtration~\( {\left( \mathscr{ F } ^{ s } _t  \right)} _{ s,t }  \). Consider three processes~\( X ,Y , Z  \), where \( X  \) is \( \ell^p ( \mathbb{ N } ) \)-valued, \(Y  \) is \( \ell^{ q } ( \mathbb{ N } ) \)-valued, and \( Z  \) is \( \mathbb{ R }  \)-valued. We assume that these are additive processes with respect to the filtration~\( (\mathscr{ F } ^{ s } _t)_{ s,t }  \). We claim that the~\( \mathscr{ G }  \)-valued process~\( (x^s_t)_{ s,t }  \) given by 
\begin{equation*}
  x^s_t= {\left( X_t - X_s, Y_t - Y_s , Z_t- Z_s + \tfrac{ 1 }{ 2 } \int_{ s<t_1<t_2 \leq t } {\left\langle \dd X _{ t_1 }  \middle\vert \dd Y_{ t_2 } \right\rangle} - {\left\langle  \dd X_{ t_2 }  \middle\vert \dd Y_{ t_1 }  \right\rangle}   \right)} 
\end{equation*}
is a multiplicative process assuming that the integrals are well-defined. It is clear that \( (x^{ s } _t)_{ s,t }  \) is stochastically continuous and adapted to the filtration~\( {\left( \mathscr{ F } ^{ s } _t \right)} _{ s,t }  \). It is left to check that \( x^{ 0 } _t = x^{ 0 } _s x^{ s } _t \) if \( 0 \leq s \leq t \). But only the last coordinate of~\( x^{ 0 } _t \) has to be checked. Notice that the last coordinate of~\( x^0_t \) is
\begin{equation*}
  Z_t  + \tfrac{ 1 }{ 2 } \int_{ 0 < t_1 < t_2 \leq  t } {\left\langle  \dd X _{ t_1 }  \middle\vert \dd Y_{ t_2 }  \right\rangle}  - {\left\langle  \dd X_{ t_2 }  \middle\vert \dd Y_{ t_1 }  \right\rangle} .
\end{equation*}
The value for~\( x^{ 0 } _s \cdot  x^{ s } _t\) is given by
\begin{align}
  x^{ 0 } _s x^{ s } _t &=  \biggl( X_t , Y_t , Z_t + \tfrac{ 1 }{ 2 } \int_{ \mathrlap{0<t_1<t_2 \leq s} } \quad {\left\langle \dd X _{ t_1 }  \middle\vert \dd Y_{ t_2 } \right\rangle}  - {\left\langle  \dd X_{ t_2 }  \middle\vert \dd Y_{ t_1 }  \right\rangle} \nonumber\\
                        & \quad + \tfrac{ 1 }{ 2 } \int_{ \mathrlap{s<t_1<t_2 \leq t }} \quad  {\left\langle  \dd X _{ t_1 }  \middle\vert \dd Y_{ t_2 } \right\rangle}  - {\left\langle  \dd X_{ t_2 }  \middle\vert\dd  Y_{ t_1 }  \right\rangle}  + \tfrac{ 1 }{ 2 } {\left( {\left\langle  X_s \middle\vert Y_t-Y_s \right\rangle}  - {\left\langle  X_t - X_s \middle\vert Y_s \right\rangle}   \right)}  \biggr) \label{Heisenberg (1)}.
\end{align}
We can write 
\begin{align*}
    {\left\langle  X_s \middle\vert Y_t-Y_s \right\rangle}  - {\left\langle X_t - X_s \middle\vert Y_s \right\rangle}  & = {\left\langle  \int_{ 0 } ^{ s } \dd X_{ t_1 }  \middle\vert \int_{ s } ^{ t } \dd Y_{ t_2 }  \right\rangle}  - {\left\langle  \int_{ s } ^{ t } \dd X_{ t_2 }  \middle\vert \int_{ 0 } ^{ s } \dd Y_{ t_1 }  \right\rangle}  \\
                                                                                                                       & = \int_{ \mathrlap{0 < t_1 \leq  s < t_2 \leq  t} } \quad {\left\langle  \dd X_{ t_1 }  \middle\vert \dd Y_{ t_2 }\right\rangle}  - {\left\langle  \dd X_{ t_2 }  \middle\vert\dd  Y_{ t_1 }  \right\rangle}  .
\end{align*}
Thus, the right-hand side in (\ref{Heisenberg (1)}) contains a sum of integrals with identical integrands over disjoint sets whose union over these sets is exactly \( \{ (t_1,t_2) : 0 < t_1 < t_2 \leq t \} \). Hence, \( x^{ 0 } _s x^{ s } _t = x^{ 0 } _t \). Similarly, as one can see that \( ( x^{ s } _t) _{ s,t } \) inherits independence of the increments; if \( X,Y,a \) are L\'evy processes, then \( (x^{ s } _t)_{ s,t }  \) becomes a multiplicative process with stationary increments, i.e., a \emph{multiplicative L\'evy process}.


\begin{thebibliography}{HDL86}

\bibitem[AM78]{abraham2008foundations}
R.~H. Abraham and J.~E. Marsden.
\newblock {\em Foundations of Mechanics}.
\newblock Benjamin/Cummings Publishing Co., Reading, 2nd edition, 1978.

\bibitem[App09]{applebaum2009levy}
D.~Applebaum.
\newblock {\em {L}\'evy Processes and Stochastic Calculus}, volume 116.
\newblock Cambridge University Press, Cambridge, 2nd edition, 2009.

\bibitem[Bas10]{bass2010measurability}
R.~F. Bass.
\newblock The measurability of hitting times.
\newblock {\em Electron. Commun. Probab.}, 15:99--105, 2010.

\bibitem[BL12]{behme2012multivariate}
A.~Behme and A.~Lindner.
\newblock Multivariate generalized {O}rnstein-{U}hlenbeck processes.
\newblock {\em Stochastic Process. Appl.}, 122(4):1487--1518, 2012.

\bibitem[CL84]{cheng2000gauge}
T.~P. Cheng and L.~F. Li.
\newblock {\em Gauge Theory of Elementary Particle Physics}.
\newblock The Clarendon Press, New York, 1984.

\bibitem[CM96]{cruzeiro1996renormalized}
A.-B. Cruzeiro and P.~Malliavin.
\newblock Renormalized differential geometry on path space: structural equation, curvature.
\newblock {\em J. Funct. Anal.}, 139(1):119--181, 1996.

\bibitem[DG08]{driver2008heat}
B.~K. Driver and M.~Gordina.
\newblock Heat kernel analysis on infinite-dimensional {H}eisenberg groups.
\newblock {\em J. Funct. Anal.}, 255(9):2395--2461, 2008.

\bibitem[Dyn47]{dynkin1947calculation}
E.~B. Dynkin.
\newblock Calculation of the coefficients in the {C}ampbell-{H}ausdorff formula.
\newblock {\em Dokl. Akad. Nauk SSSR}, 57:323--326, 1947.

\bibitem[Dyn50]{dynkin1950normed}
E.~B. Dynkin.
\newblock Normed {L}ie algebras and analytic groups.
\newblock {\em Uspekhi Matem. Nauk}, 5(1(35)):135--186, 1950.

\bibitem[Ebe94]{eberlein1994geometry}
P.~Eberlein.
\newblock Geometry of {$2$}-step nilpotent groups with a left invariant metric. {II}.
\newblock {\em Trans. Amer. Math. Soc.}, 343(2):805--828, 1994.

\bibitem[{\'E}me89]{emery1989stochastic}
M.~{\'E}mery.
\newblock {\em Stochastic Calculus in Manifolds}.
\newblock Springer, Berlin, 1989.

\bibitem[Est92]{estrade1992exponentielle}
A.~Estrade.
\newblock Exponentielle stochastique et int\'egrale multiplicative discontinues.
\newblock {\em Ann. Inst. H. Poincar\'e Probab. Statist.}, 28(1):107--129, 1992.

\bibitem[FK85]{fujiwara1985stochastic}
T.~Fujiwara and H.~Kunita.
\newblock Stochastic differential equations of jump type and {L}\'evy processes in diffeomorphisms group.
\newblock {\em J. Math. Kyoto Univ.}, 25(1):71--106, 1985.

\bibitem[GS04]{gikhman2004theory}
I.~I. Gikhman and A.~V. Skorokhod.
\newblock {\em The Theory of Stochastic Processes {II}}.
\newblock Springer, Berlin, 2004.

\bibitem[HDL86]{hakim2006exponentielle}
M.~Hakim-Dowek and D.~L\'epingle.
\newblock L'exponentielle stochastique des groupes de {L}ie.
\newblock In {\em S{\'e}minaire de Probabilit{\'e}s XX 1984/85}, pages 352--374, Berlin, Heidelberg, 1986. Springer.

\bibitem[Hey77]{heyer1977probability}
H.~Heyer.
\newblock {\em Probability Measures on Locally Compact Groups}, volume Band 94 of {\em Ergebnisse der Mathematik und ihrer Grenzgebiete [Results in Mathematics and Related Areas]}.
\newblock Springer-Verlag, Berlin-New York, 1977.

\bibitem[Hun56]{hunt1956semigroups}
G.~A. Hunt.
\newblock Semi-groups of measures on {L}ie groups.
\newblock {\em Trans. Amer. Math. Soc.}, 81:264--293, 1956.

\bibitem[It{\^o}04]{ito2013stochastic}
K.~It{\^o}.
\newblock {\em Stochastic Processes}.
\newblock Springer, Berlin, Heidelberg, 1st edition, 2004.

\bibitem[Kac90]{kac1990infinite}
V.~G. Kac.
\newblock {\em Infinite-Dimensional {L}ie Algebras}.
\newblock Cambridge University Press, Cambridge, 3rd edition, 1990.

\bibitem[Kov93]{kovalchuk1993semimartingale}
L.~V. Koval'chuk.
\newblock Semimartingales with values in {L}ie groups and algebras.
\newblock {\em Ukra\"in. Mat. Zh.}, 45(2):251--257, 1993.

\bibitem[Kun97]{kunita1990stochastic}
H.~Kunita.
\newblock {\em Stochastic Flows and Stochastic Differential Equations}, volume~24 of {\em Cambridge Studies in Advanced Mathematics}.
\newblock Cambridge University Press, Cambridge, 1997.
\newblock Reprint of the 1990 original.

\bibitem[Lan99]{lang1999fundamentals}
S.~Lang.
\newblock {\em Fundamentals of Differential Geometry}, volume 191.
\newblock Springer, New York, 1999.

\bibitem[Lia04]{liao2004levy}
M.~Liao.
\newblock {\em {L}\'evy Processes in {L}ie Groups}, volume 162.
\newblock Cambridge University Press, Cambridge, 2004.

\bibitem[MZ55]{montgomery1955topological}
D.~Montgomery and L.~Zippin.
\newblock {\em Topological Transformation Groups}.
\newblock Interscience Publishers, New York-London, 1955.

\bibitem[Nee06]{neeb2015towards}
K.-H. Neeb.
\newblock Towards a {L}ie theory of locally convex groups.
\newblock {\em Jpn. J. Math.}, 1(2):291--468, 2006.

\bibitem[Sko82]{skorokhod1982operator}
A.~V. Skorokhod.
\newblock Operator stochastic differential equations and stochastic semigroups.
\newblock {\em Uspekhi Mat. Nauk}, 37(6):157--183, 1982.

\bibitem[Var84]{varadarajan2013lie}
V.~S. Varadarajan.
\newblock {\em {Lie} Groups, {L}ie Algebras, and their Representations}, volume 102.
\newblock Springer, New York, 1984.

\bibitem[vEK64]{est1964non}
W.~T. van Est and T.~J. Korthagen.
\newblock Non-enlargible {L}ie algebras.
\newblock {\em Indag. Math.}, 26:15--31, 1964.

\bibitem[Zor45]{zorn1945characterization}
M.~A. Zorn.
\newblock Characterization of analytic functions in {B}anach spaces.
\newblock {\em Ann. of Math. (2)}, 46:585--593, 1945.

\end{thebibliography}
\end{document}